\documentclass[a4paper,11pt]{article}
\usepackage[latin1]{inputenc}
\usepackage[english]{babel}
\usepackage{amsmath}
\usepackage{amsfonts}
\usepackage{amssymb}
\usepackage{epsfig}
\usepackage{amsopn}
\usepackage{amsthm}
\usepackage{color}
\usepackage{graphicx}
\usepackage{subfigure}
\usepackage{enumerate}
\newtheorem{theorem}{Theorem}[section]

\newtheorem{lemma}[theorem]{Lemma}
\newtheorem{proposition}[theorem]{Proposition}
\newtheorem{definition}[theorem]{Definition}

\newtheorem*{theorem*}{Theorem}
\newtheorem*{lemma*}{Lemma}
\newtheorem*{remark*}{Remark}
\newtheorem*{definition*}{Definition}
\newtheorem*{proposition*}{Proposition}
\newtheorem*{corollary*}{Corollary}
\numberwithin{equation}{section}
\newcommand{\real}{\mathbb{R}}

\def\qed{\,\unskip\kern 6pt \penalty 500
\raise -2pt\hbox{\vrule \vbox to8pt{\hrule width 6pt
\vfill\hrule}\vrule}\par}
\definecolor{darkblue}{rgb}{0.05, .05, .65}
\definecolor{darkgreen}{rgb}{0.1, .65, .1}
\definecolor{darkred}{rgb}{0.8,0,0}
\newcommand{\beqn}{\begin{equation}}
\newcommand{\eeqn}{\end{equation}}
\newcommand{\bear}{\begin{eqnarray}}
\newcommand{\eear}{\end{eqnarray}}
\newcommand{\bean}{\begin{eqnarray*}}
\newcommand{\eean}{\end{eqnarray*}}
\begin{document}

\title{\huge \bf Self-similar blow-up profiles for a reaction-diffusion equation with critically strong weighted reaction}

\author{
\Large Razvan Gabriel Iagar\,\footnote{Departamento de Matem\'{a}tica
Aplicada, Ciencia e Ingenieria de los Materiales y Tecnologia
Electr\'onica, Universidad Rey Juan Carlos, M\'{o}stoles,
28933, Madrid, Spain, \textit{e-mail:} razvan.iagar@urjc.es},
\\[4pt] \Large Ariel S\'{a}nchez,\footnote{Departamento de Matem\'{a}tica
Aplicada, Ciencia e Ingenieria de los Materiales y Tecnologia
Electr\'onica, Universidad Rey Juan Carlos, M\'{o}stoles,
28933, Madrid, Spain, \textit{e-mail:} ariel.sanchez@urjc.es}\\
[4pt] }
\date{}
\maketitle

\begin{abstract}
We classify the self-similar blow-up profiles for the following reaction-diffusion equation with critical strong weighted reaction and unbounded weight:
$$
\partial_tu=\partial_{xx}(u^m) + |x|^{\sigma}u^p,
$$
posed for $x\in\real$, $t\geq0$, where $m>1$, $0<p<1$ such that $m+p=2$ and $\sigma>2$ completing the analysis performed in a recent work where this very interesting critical case was left aside. We show that finite time blow-up solutions in self-similar form exist for $\sigma>2$. Moreover all the blow-up profiles have compact support and their supports are \emph{localized}: there exists an explicit $\eta>0$ such that any blow-up profile satisfies ${\rm supp}\,f\subseteq[0,\eta]$. This property is unexpected and contrasting with the range $m+p>2$. We also classify the possible behaviors of the profiles near the origin.
\end{abstract}

\


\smallskip


\section{Introduction}

This paper is devoted to the study of the self-similar blow-up profiles for the following reaction-diffusion equation with weighted reaction
\begin{equation}\label{eq1}
u_t=(u^m)_{xx}+|x|^{\sigma}u^p, \qquad u=u(x,t), \quad
(x,t)\in\real\times(0,T),
\end{equation}
with exponents $1<m<2$, $0<p<1$ such that $m+p=2$ and $\sigma>2$. This work completes the analysis performed in the recent paper \cite{IS20b}, where blow-up profiles were obtained and analyzed for the range $m>1$, $0<p<1$ but with $m+p\neq2$. Let us stress at this point that the case we are dealing with here is a critical one as being a limit case between the ranges of exponents $m+p>2$ (characterized by a big variety of compactly supported blow-up profiles, with two different types of interface) and $m+p<2$ (where complete non-existence of blow-up profiles is proved). Thus, some interesting differences with respect to the two well-understood cases are expected to hold true when $m+p=2$ and we show in the present work that this is indeed the case.

Starting from the famous paper by Fujita \cite{Fu66}, studying the finite time blow-up phenomenon for solutions to reaction-diffusion equations has become a very interesting and developing line of research in the field of parabolic partial differential equations and many of the outstanding experts of the field contributed to this research. To fix the ideas and notation, we say that a solution $u$ to Eq. \eqref{eq1} \emph{blows up in finite time} if there exists $T\in(0,\infty)$ such that $u(T)\not\in L^{\infty}(\real)$, but $u(t)\in L^{\infty}(\real)$ for any $t\in(0,T)$. The smallest time $T<\infty$ satisfying this property is known as the blow-up time of $u$. We also use throughout the paper the (well-established) notation $u(t)$ for the map $x\mapsto u(x,t)$ at a fixed time $t\in[0,T]$.

At first, research on finite time blow-up focused on establishing for which initial conditions the solutions to the Cauchy or Dirichlet problems presented finite time blow-up. This is now well-understood for many of the most interesting model equations. Later on, mathematicians begun to raise finer problems related to the phenomenon of finite-time blow-up, such as studying at which points do the solutions become infinite (the blow-up set) and also estimating the dynamics of the solutions near the blow-up time, in the form of both blow-up rates, that is sharp estimates for $\|u(t)\|_{\infty}$ as $t$ approaches the blow-up time $T$, and blow-up patterns (or profiles), that is special solutions with symmetries towards which many general solutions converge as $t\to T$. All this long term research program has been performed at first for the prototype reaction-diffusion equation
\begin{equation}\label{eq1.hom}
u_t=\Delta u^m+u^p, \qquad m\geq1, \ p>1,
\end{equation}
in one dimension and later in higher space dimension, where a number of questions are still open. We have nowadays two well-written monographs on the blow-up phenomenon for Eq. \eqref{eq1.hom}, one dealing with the semilinear case $m=1$ \cite{QS} and another with the quasilinear case $m>1$ \cite{S4}. In particular, related to the goal of our paper, the \emph{relevance of the self-similar blow-up profiles} to Eq. \eqref{eq1.hom} in the development of the general theory is evident from, for example, \cite[Chapter 4]{S4}, the study of the blow-up profiles being later extended in \cite{GV97}.

The analysis of the finite time blow-up for reaction-diffusion equations with \emph{weighted reaction}, that is (in the most general form)
\begin{equation}\label{eq1.gen}
u_t=\Delta u^m+V(x)u^p,
\end{equation}
where $V(x)$ is a function of $x$ with suitable properties, started later due to the difficulty of the problem and to the fact that some of the techniques used in the study of Eq. \eqref{eq1.hom} do no longer apply for equations with non-constant coefficients (noticeably, tools such as intersection comparison or the translation invariance of the equation). In the semilinear case $m=1$ works by Baras and Kersner \cite{BK87}, Bandle and Levine \cite{BL89} completed by Levine and Meier \cite{LM90} on cones and later Pinsky \cite{Pi97, Pi98} established the Fujita exponent $p^*$, that is the largest exponent $p>1$ such that for any $p\in(1,p^*)$ all the non-trivial solutions blow up in finite time, for the power weight $V(x)=|x|^{\sigma}$ or more general unbounded weights and studied some further properties of the solutions. The same problem of establishing the Fujita exponent has been addressed by Suzuki \cite{Su02} for $m>1$ and $p>m$, and for weights $V(x)$ such that $V(x)=|x|^{\sigma}$ at least for $|x|>R$ large. Moreover, in the same paper Suzuki gives a criterion for finite time blow-up in terms of the behavior of the initial condition $u_0(x)$ as $|x|\to\infty$ for the remaining range $p>p^*$.

Concerning the qualitative theory, Andreucci and DiBenedetto perform in \cite{AdB91} a complete study of local existence, initial traces and Harnack inequalities for Eq. \eqref{eq1.gen} with $m>1$, $p>1$ and $V(x)=(1+|x|)^{-\sigma}$ for any $\sigma\in\real$, thus bringing the basic theory for the Cauchy problem with both bounded and unbounded weights. With similar techniques, Andreucci and Tedeev \cite{AT05} obtained the blow-up rate of solutions for a class of more general parabolic equations containing \eqref{eq1.gen} as a particular case, again for $p>m>1$ and for weights $V(x)=|x|^{\sigma}$ for $m<p<m+2/N$ and $0<\sigma\leq N(p-m)/m$, where most probably the upper bound for $\sigma$ is just a technical limitation. More recently, a number of problems related to the finite time blow-up of solutions to Eq. \eqref{eq1.gen} have been investigated for compactly supported weights $V(x)$ starting from the work by Ferreira, de Pablo and V\'azquez \cite{FdPV06} in the one-dimensional case. These results were then extended to $\real^N$ in \cite{KWZ11, Liang12} and also to the fast diffusion case $m<1$ by \cite{BZZ11}. Another problem thoroughly studied recently, concerning the blow-up sets, is to decide whether the zeros of the weight $V(x)$ can be a blow-up point or not for a solution to Eq. \eqref{eq1.gen}. This has been considered in a series of works dealing mostly with the case of the homogeneous Dirichlet problem in a bounded domain, such as \cite{GLS10, GLS13, GS11, GS18}.

The qualitative theory and the study of the dynamics for solutions to Eq. \eqref{eq1.hom} with $m>1$ but $0<p<1$ has been developed in a series of papers by de Pablo and V\'azquez \cite{dPV90, dPV91, dPV92, dP94} in which the authors prove (in dimension one) that the Cauchy problem for Eq. \eqref{eq1.hom} is ill posed as uniqueness of solutions is lacking. It is in fact proved that for any initial condition $u_0$ in suitable spaces, there exist a minimal and a maximal solution, and some criteria for uniqueness are given strongly depending on the sign of the expression $m+p-2$: when $m+p\geq2$ uniqueness holds true if and only if $u_0(x)>0$ for any $x\in\real$, while for $m+p<2$ uniqueness holds true for any data $u_0\not\equiv0$. This and the difference with respect to the speed of propagation (finite if $m+p\geq2$, infinite if $m+p<2$) show us that the sign of the number $m+p-2$ is critical for the equation.

Coming back to our Eq. \eqref{eq1}, the authors started recently a long-term project to study the influence of the unbounded weight $|x|^{\sigma}$ on the blow-up phenomenon. Taking into account the relevance of the self-similar blow-up patterns for the general dynamics, we begun with the study of them and we have thus proved in previous works devoted to the range $p\geq1$ \cite{IS19a, IS19b, IS20a} that the magnitude of $\sigma$ has a \emph{strong influence on the form and blow-up set} of the solutions. Even stronger than that, for the case $p=m>1$ \cite{IS20a}, multiple blow-up profiles may exist for $\sigma$ small but for $\sigma$ sufficiently large they even cease to exist. Entering the range $0<p<1$, we classified the blow-up profiles to Eq. \eqref{eq1} with $\sigma>2(1-p)/(m-1)$ and $m+p\neq2$ in our previous paper \cite{IS20b}, obtaining among other results the following most interesting ones:

$\bullet$ with $\sigma>2(1-p)/(m-1)$ finite time blow-up occurs for compactly supported self-similar solutions, a fact that is due to the presence of the weight. For $\sigma=0$ compactly supported solutions to Eq. \eqref{eq1} are global.

$\bullet$ for $m+p>2$, there are many blow-up profiles, presenting two different interface behaviors, one inherited from the porous medium equation and one inherited only from the combination between $u_t$ and the reaction term. For any fixed $\sigma>2(1-p)/(m-1)$ there are infinitely many profiles with the second type of interface and at least one (we conjecture that exactly one) profile with the first type of interface.

$\bullet$ for $m+p<2$ there are no blow-up self-similar solutions at all. We strongly believe that this has to do with a general non-existence (in the form of an instantaneous global blow-up) of \emph{any non-trivial solution} for $m+p<2$. We prove in fact in the forthcoming paper \cite{IMS20} complete non-existence of nonzero solutions to a similar and related equation to Eq. \eqref{eq1}, more precisely for Eq. \eqref{eq1.gen} with $V(x)=(1+|x|)^{\sigma}$ for $\sigma>2(1-p)/(m-1)$.

The present paper is aimed to complete the work done in \cite{IS20b} by addressing the remaining case, that is when $m+p=2$ and $\sigma>2$, which plays the role of an interface between the two ranges explained above. As we shall see, significant differences appear both with respect to the results and to some of the methods in the proofs. This will be detailed in the following paragraphs where we expose our main results.

\bigskip

\noindent \textbf{Main results}. As already mentioned, the present work is devoted to the self-similar blow-up profiles for Eq. \eqref{eq1}. Their big relevance for the general theory of the equations is well-established as shown for example for the standard porous medium equation (where the theory has been developed having the Barenblatt solutions as a core \cite{VPME}) or the quasilinear reaction-diffusion equations without weights \cite[Chapter 4]{S4}. Let us add that the results obtained in the current work and in \cite{IS20b} will be strongly used in the forthcoming paper \cite{IMS20} on the qualitative theory of solutions to a similar equation. We already expect from the study performed in \cite{IS20b} that for $\sigma>2$ finite time blow-up occurs, thus we look for \emph{backward self-similar solutions} in the form
\begin{equation}\label{SSform}
u(x,t)=(T-t)^{-\alpha}f(\xi), \quad \xi=|x|(T-t)^{\beta},
\end{equation}
where $T\in(0,\infty)$ is the finite blow-up time and $\alpha>0$, $\beta\in\real$ exponents to be determined. Since $p=2-m$ we get by replacing the ansatz \eqref{SSform} in \eqref{eq1} that
\begin{equation}\label{SSexp}
\alpha=\frac{\sigma+2}{\sigma(m-1)+2(p-1)}=\frac{\sigma+2}{(\sigma-2)(m-1)}, \quad \beta=\frac{m-p}{\sigma(m-1)+2(p-1)}=\frac{2}{\sigma-2}
\end{equation}
and the self-similar profile $f$ is a solution to the non-autonomous differential equation
\begin{equation}\label{SSODE}
(f^m)''(\xi)-\alpha f(\xi)+\beta\xi f'(\xi)+\xi^{\sigma}f^{2-m}(\xi)=0, \quad \xi\in[0,\infty).
\end{equation}
It is easy to see that the lower bound for $\sigma$ is needed in order to have blow-up solutions, a fact which is in \emph{strong contrast} with, for example, the autonomous case $\sigma=0$ where solutions exist and remain bounded globally in time, as shown for example in \cite{dPV90, dPV91}. We thus conclude that the weight plays an essential role for finite time blow-up to occur. We define what we understand by a \emph{good profile} below (similar to \cite{IS19a, IS20b}).
\begin{definition}\label{def1}
We say that a solution $f$ to the differential equation \eqref{SSODE} is a \textbf{good profile} if it fulfills one of the following two properties related to its behavior at $\xi=0$:

\indent (P1) $f(0)=a>0$, $f'(0)=0$.

\indent (P2) $f(0)=0$, $(f^m)'(0)=0$.

\noindent We say that a profile $f$ has an \textbf{interface} at some point $\xi_0\in(0,\infty)$ if
$$
f(\xi_0)=0, \qquad (f^m)'(\xi_0)=0, \qquad f>0 \ {\rm on} \ (\xi_0-\delta,\xi_0), \ {\rm for \ some \ } \delta>0.
$$
\end{definition}
This definition agrees with the well-known notion of an interface for an evolutionary solution. Indeed, a solution $u$ to Eq. \eqref{eq1} in the form \eqref{SSform} with a profile $f$ having an interface at $\xi_0\in(0,\infty)$, has a time-moving interface at $|x|=s(t)=(T-t)^{-\beta}\xi_0$ for any $t\in(0,T)$. Throughout the paper we are interested in the \emph{good profiles with interface} according to Definition \ref{def1}. For $m+p>2$ we proved in \cite{IS20b} that there are two different interface behaviors and that any $\xi_0\in(0,\infty)$ can be the interface point for at least a profile. In our case things are very different. First of all, the two interface behaviors whose existence has been proved in \cite{IS20b}, namely
\begin{equation}\label{beh.interf}
f(\xi)\sim C(\xi_0-\xi)^{1/(m-1)}, \ {\rm respectively} \ f(\xi)\sim C(\xi_0-\xi)^{1/(1-p)}, \ {\rm as} \ \xi\to\xi_0\in(0,\infty)
\end{equation}
join for $m+p=2$, thus we have a \emph{single interface behavior} given by any of the two asymptotic expressions in \eqref{beh.interf}. But even more striking is the fact that \emph{the set of points $\xi_0\in(0,\infty)$} that can be points of interface for a profile \emph{is localized}. Introducing for a fixed $\sigma>2$ the following constant
\begin{equation}\label{ximax}
\xi_{max}:=\left(\frac{\beta^2}{4m}\right)^{1/(\sigma-2)}=\left(\frac{1}{m(\sigma-2)^2}\right)^{1/(\sigma-2)}\in(0,\infty),
\end{equation}
we can state our existence result for blow-up profiles in the case $m+p=2$.
\begin{theorem}\label{th.exist}
Given $m$, $p$ and $\sigma$ such that $m+p=2$ and $\sigma>2$, for any $\xi_0\in(0,\xi_{max}]$, there exists at least a good profile with interface $f(\xi)$ in the sense of Definition \ref{def1} having its interface point exactly at $\xi=\xi_0$. For any $\xi_0\in(\xi_{max},\infty)$ there is no good profile with interface exactly at $\xi_0$.
\end{theorem}
This result is in \emph{striking contrast} with the range of exponents $p\in(1,m)$, $p=1$ or $p\in(0,1)$ with $m+p>2$, analyzed in previous works \cite{IS19a, IS19b, IS20b}, where any $\xi_0\in(0,\infty)$ can be an interface point.

\medskip

\noindent \textbf{A formal calculation to understand Theorem \ref{th.exist}}. The "mysterious" localization of $\xi_0$ in Theorem \ref{th.exist} can be understood from the equation of the profiles \eqref{SSODE} by the following formal calculation. Indeed, letting $g(\xi)=mf^{m-1}(\xi)/(m-1)$ (the pressure variable), the equation of profiles becomes
\begin{equation}\label{pressure}
(m-1)g(\xi)g''(\xi)+(g')^2(\xi)-(m-1)\alpha g(\xi)+\beta\xi g'(\xi)+m\xi^{\sigma}\left(\frac{m-1}{m}g(\xi)\right)^{(m+p-2)/(m-1)}=0.
\end{equation}
Assuming at a formal level (satisfied nevertheless rigorously by the self-similar solutions) that the terms in $g$ and $g''g$ vanish at the given interface point $\xi_0\in(0,\infty)$ and that $g'(\xi_0)\neq0$ (as it follows from \eqref{beh.interf}), we notice the reason of this striking contrast between the cases $m+p>2$ and $m+p=2$. In the former, since $m+p>2$, the final term in \eqref{pressure} vanishes too at $\xi=\xi_0$, whence
$$
g'(\xi_0)+\beta\xi_0=0,
$$
which gives the derivative of $g$ at the interface point $\xi=\xi_0$ and no limitation on $\xi_0\in(0,\infty)$, as seen in \cite{IS20b}. Meanwhile in the latter, since $m+p=2$, the final term does not vanish and evaluating \eqref{pressure} at $\xi=\xi_0$ we are left with
$$
(g'(\xi_0))^2+\beta\xi_0 g'(\xi_0)+m\xi_0^{\sigma}=0,
$$
which is a second degree equation on $g'(\xi_0)$ whose discriminant is
$$
\Delta=\beta^2\xi_0^2-4m\xi_0^{\sigma}.
$$
The requirement for the existence of a solution to \eqref{pressure} with interface at the given point $\xi=\xi_0$ translates into $\Delta\geq0$, that is $\xi_0\in(0,\xi_{max}]$ as stated in Theorem \ref{th.exist}.

\medskip

\noindent \textbf{Classification of the profiles}. Since the interface behavior at any $\xi_0\in(0,\xi_{max}]$ is the same as in \eqref{beh.interf}, we can classify the good profiles with interface by their behavior as $\xi\to0$. As also seen in previous papers \cite{IS19a, IS20b} the good profiles solutions to \eqref{SSODE} can have three different types of behavior as $\xi\to0$ and they are very significant with respect to the blow up of the corresponding self-similar solutions given by \eqref{SSform}:

$\bullet$ profiles satisfying assumption (P1) in Definition \ref{def1}, and the corresponding self-similar solutions given by \eqref{SSform} blow up globally (that is, simultaneously at every $x\in\real$) as $t\to T$, as shown by the following calculation
\begin{equation*}
u(x,t)=(T-t)^{-\alpha}f(|x|(T-t)^{\beta})\sim a(T-t)^{-\alpha}, \quad {\rm as} \ t\to T.
\end{equation*}

$\bullet$ profiles satisfying assumption (P2) in Definition \ref{def1} with the specific behavior
\begin{equation}\label{beh.P2}
f(\xi)\sim\left[\frac{m-1}{2m(m+1)}\right]^{1/(m-1)}\xi^{2/(m-1)}, \qquad {\rm as} \ \xi\to0,
\end{equation}
and the corresponding self-similar solutions given by \eqref{SSform} also blow up globally as $t\to T$, as shown by the following calculation at every $x\in\real$
\begin{equation*}
u(x,t)\sim C(T-t)^{-\alpha+2\beta/(m-1)}|x|^{2/(m-1)}=C(T-t)^{-1/(m-1)}|x|^{2/(m-1)}, \quad {\rm as} \ t\to T.
\end{equation*}

$\bullet$ profiles satisfying assumption (P2) in Definition \ref{def1} with the specific behavior
\begin{equation}\label{beh.P0}
f(\xi)\sim K\xi^{(\sigma+2)/(m-p)}, \ \ K>0, \qquad {\rm as} \ \xi\to0,
\end{equation}
and the corresponding self-similar solutions given by \eqref{SSform} remain bounded at every fixed $x\in\real$ as it follows from
$$
u(x,t)\sim C(T-t)^{-\alpha+(\sigma+2)\beta/(m-p)}|x|^{(\sigma+2)/(m-p)}=C|x|^{(\sigma+2)/(m-p)}<\infty, \quad {\rm as} \ t\to T.
$$
Such profiles blow up at $t=T$ \emph{only at the space infinity} in the following sense \cite{La84, GU06}: $\|u(t)\|_{\infty}\to\infty$ as $t\to T$ but the maximum attains on curves $x(t)$ depending on $t$ such that $x(t)\to\infty$ as $t\to T$.

With respect to these possible behaviors of the good profiles at $\xi=0$, we classify the good profiles with interface in the following long statement gathering all the possible cases.
\begin{theorem}\label{th.class}
Let $m>1$, $p\in(0,1)$ be such that $m+p=2$ and $\sigma>2$. We have:

(a) For any $\sigma\in(2,\infty)$ there exist good profiles with interface behaving as in \eqref{beh.P0} as $\xi\to0$. The corresponding self-similar solutions blow up in finite time only at space infinity.

(b) There exists $\sigma_0>2$ such that for any $\sigma\in(2,\sigma_0)$ there exist good profiles with interface and with any of the three possible behaviors as $\xi\to0$.

(c) For any fixed $\xi_0\in(0,\xi_{max}]$ there exists an exponent $\sigma(\xi_0)>2$ (depending on $\xi_0$) such that with $\sigma=\sigma(\xi_0)$ there exists a good profile with interface exactly at $\xi=\xi_0$ and with the behavior given by \eqref{beh.P2} as $\xi\to0$.

(d) There exists $\sigma_1>2$ such that for any $\sigma\in(\sigma_1,\infty)$ there are no good profiles with interface with behavior given by \eqref{beh.P2} as $\xi\to0$.
\end{theorem}

\noindent \textbf{A comment on the techniques}. Let us notice first that the main technique we used in \cite{IS19a, IS20b}, that is, the backward shooting method from the interface point, is no longer possible here. Indeed, in the former cases any $\xi_0\in(0,\infty)$ was an interface point and for every fixed $\xi_0\in(0,\infty)$ there was an unique profile with interface exactly at that point (at least for interfaces of Type I in \cite{IS20b}). This is not the case here, as Theorem \ref{th.exist} states: for $\xi_0>\xi_{max}$ there are no profiles with interface there, while for $\xi_0\in(0,\xi_{max})$ there are many of them, thus shooting cannot be performed from any $\xi_0>0$. We will use in the proofs the general technique of a phase space analysis associated to an autonomous quadratic system of differential equations, but the proofs will be done directly using \emph{the geometry of the phase space} and not an analysis in terms of profiles. We stress here that some geometrical arguments in the phase space are quite involved and based on constructing local barriers for the trajectories of the system in form of suitable planes, surfaces or combinations of them that the orbits cannot cross. Moreover, we notice that every fixed $\xi_0\in(0,\xi_{max})$ encodes a classification of profiles with prescribed interface at $\xi=\xi_0$, a fact that is new with respect to the case $m+p>2$ as seen in \cite{IS20b}.

\section{The phase space and the critical parabola}\label{sec.local}

Following the ideas in \cite{IS20b} we transform the differential equation of the profiles \eqref{SSODE} into an autonomous, quadratic dynamical system by letting
\begin{equation}\label{PSvar1}
X(\eta)=\frac{m}{\alpha}\xi^{-2}f^{m-1}(\xi), \ Y(\eta)=\frac{m}{\alpha}\xi^{-1}f^{m-2}(\xi)f'(\xi), \ Z(\eta)=\frac{m}{\alpha^2}\xi^{\sigma-2},
\end{equation}
where $\alpha$ (and also $\beta$) is defined in \eqref{SSexp} and the new independent variable $\eta=\eta(\xi)$ is defined through the differential equation
$$
\frac{d\eta}{d\xi}=\frac{\alpha}{m}\xi f^{1-m}(\xi).
$$
We are thus left with the system
\begin{equation}\label{PSsyst1}
\left\{\begin{array}{ll}\dot{X}=X[(m-1)Y-2X],\\
\dot{Y}=-Y^2-\frac{\beta}{\alpha}Y+X-XY-Z,\\
\dot{Z}=(\sigma-2)XZ,\end{array}\right.
\end{equation}
which is similar to the one in \cite[Section 2]{IS20b} but with some noticeable differences: first, since in our new variables $X\geq0$ and $Z\geq0$ (only $Y$ is allowed to change sign), we infer from the third equation that variable $Z$ is non-decreasing along the trajectories of the system. Notice also that the planes $\{X=0\}$ and $\{Z=0\}$ are invariant for the system \eqref{PSsyst1}. Moreover, an easy inspection of the system gives that the critical points in the plane are
$$
P_0^{\lambda}=\left(0,\lambda,-\lambda^2-\frac{\beta}{\alpha}\lambda\right), \ \lambda\in\left[-\frac{\beta}{\alpha},0\right] \qquad P_2=\left(X(P_2),Y(P_2),0\right),
$$
where
\begin{equation}\label{not.P2}
X(P_2)=\frac{m-1}{2(m+1)\alpha}=\frac{(m-1)^2(\sigma-2)}{2(m+1)(\sigma+2)}, \ Y(P_2)=\frac{1}{(m+1)\alpha}=\frac{(m-1)(\sigma-2)}{(m+1)(\sigma+2)},
\end{equation}
and we are thus left with a full critical parabola of equation
\begin{equation}\label{parabola}
-Y^2-\frac{\beta}{\alpha}Y-Z=0, \qquad -\frac{\beta}{\alpha}\leq Y\leq0, \qquad X=0,
\end{equation}
and this is completely new with respect to the analogous analysis for the range $m+p>2$. The analysis of the critical points on this parabola will be different than the study of the single critical points $P_0=(0,0,0)$ and $P_1=(0,-\beta/\alpha,0)$ as done for $m+p>2$. In fact these two points are in our case the endpoints of the parabola but due to the monotonicity of the $Z$ component on the trajectories, the orbits entering them are contained in the invariant plane $\{Z=0\}$ and do not contain profiles. Let us still keep the notation $P_0=(0,0,0)$ for simplicity (instead of $P_0^0$).

\medskip

\noindent \textbf{Local analysis of the critical point in the plane}. We analyze locally the trajectories of the system \eqref{PSsyst1} in a neighborhood of the critical points $P_0$, $P_0^{\lambda}$ and $P_2$.
\begin{lemma}\label{lem.P0}
The system \eqref{PSsyst1} in a neighborhood of the critical point $P_0$ has a one-dimensional stable manifold and a two-dimensional center manifold. The connections in the plane tangent to the center manifold go out of the point $P_0$ and contain profiles with the behavior \eqref{beh.P0} for any $K>0$.
\end{lemma}
\begin{proof}
The linearization of the system \eqref{PSsyst1} near $P_0$ has the matrix
$$
M(P_0)=\left(
      \begin{array}{ccc}
        0 & 0 & 0 \\
        1 & -\frac{\beta}{\alpha} & -1 \\
        0 & 0 & 0 \\
      \end{array}
    \right)
$$
thus it has a one-dimensional stable manifold and a two-dimensional center manifold. The stable manifold is contained in the invariant plane $\{Z=0\}$ and it does not contain profiles. As for the center manifold, we can follow identically the analysis performed in \cite[Section 2]{IS20b} and based on the Local Center Manifold Theorem \cite[Theorem 1, Section 2.12]{Pe} by letting $T:=(\beta/\alpha)Y-X+Z$ to get the system in variables $(X,T,Z)$
\begin{equation}\label{interm0}
\left\{\begin{array}{ll}\dot{X}&=\frac{1}{\beta}X[X+(m-1)\alpha T-(m-1)\alpha Z],\\
\dot{T}&=-\frac{\beta}{\alpha}T-\frac{\alpha}{\beta}T^2-\frac{\alpha(m+1)+\beta}{\beta}XT+\frac{\alpha(m+p)}{\beta}TZ\\
&-\frac{m\alpha-\beta}{\beta}X^2+\frac{3\beta+2\alpha+3}{\beta}XZ-\frac{\alpha(m+p-1)}{\beta}Z^2,\\
\dot{Z}&=\frac{1}{\beta}Z[2X+(m+p-2)\alpha T-(m+p-2)\alpha Z],\end{array}\right.
\end{equation}
and find that the center manifold is given near the origin by the equation $T=0$ (up to the second order) and the flow on the center manifold is given by the system
\begin{equation}\label{interm1}
\left\{\begin{array}{ll}\dot{X}&=\frac{1}{\beta}X[X-(m-1)\alpha Z]+O(|(X,Z)|^3),\\
\dot{Z}&=\frac{2}{\beta}XZ+O(|(X,Z)|^3),\end{array}\right.
\end{equation}
whose trajectories are tangent to the explicit family
$$
X=K\sqrt{Z}-(m-1)\alpha Z, \qquad K\in\real
$$
which contain profiles satisfying \eqref{beh.P0} as readily seen from \eqref{PSvar1}. We omit the details and we refer the reader to \cite[Section 2]{IS20b} for the full calculations.
\end{proof}
It is now the turn to analyze the local behavior of the system near the points on the critical parabola, and this is the main novelty in this section with respect to the analysis done in \cite{IS20b}. We will have to make a distinction with respect to the value of $\lambda$ as in the next statement.
\begin{lemma}\label{lem.P1}
(a) For $\lambda\in(-\beta/\alpha,-\beta/2\alpha)$ the system \eqref{PSsyst1} near the critical point $P_0^{\lambda}$ has a one-dimensional center manifold, a one-dimensional unstable manifold and a one-dimensional stable manifold. The center manifold is contained in the parabola \eqref{parabola}, the unstable manifold is contained in the invariant plane $\{X=0\}$ and there is a unique orbit entering $P_0^{\lambda}$ from the half-space $\{X>0\}$.

(b) For $\lambda=-\beta/2\alpha$ the system \eqref{PSsyst1} near the critical point $P_0^{\lambda}$ has a two-dimensional center manifold and a  one-dimensional stable manifold. The rather complex behavior of the orbits entering this critical point will be analyzed later in Proposition \ref{prop.max}, leading to an interesting example of a \emph{center-stable two-dimensional manifold}.

(c) For $\lambda\in(-\beta/2\alpha,0)$ the system \eqref{PSsyst1} near the critical point $P_0^{\lambda}$ has a one-dimensional center manifold and a two-dimensional stable manifold. The center manifold is contained in the parabola \eqref{parabola}.

In all the three cases, the orbits entering $P_0^{\lambda}$ on the stable manifold contain profiles with an interface at some $\xi_0\in(0,\xi_{max}]$ and with the interface behavior \eqref{beh.interf}.
\end{lemma}
This lemma shows that our main objects of interest will be throughout the paper the orbits entering the critical parabola \eqref{parabola}.
\begin{proof}
The linearization of the system \eqref{PSsyst1} near $P_0^{\lambda}$ has the matrix
$$
M(P_0^{\lambda})=\left(
      \begin{array}{ccc}
        (m-1)\lambda & 0 & 0 \\
        1-\lambda & -2\lambda-\frac{\beta}{\alpha} & -1 \\
        (\sigma-2)\left(-\lambda^2-\frac{\beta}{\alpha}\lambda\right) & 0 & 0 \\
      \end{array}
    \right)
$$
with eigenvalues
$$
l_1=(m-1)\lambda<0, \qquad l_2=-2\lambda-\frac{\beta}{\alpha}, \qquad l_3=0.
$$
The sign of $l_2$ decides whether we are in the case (a), (b) or (c) and gives the dimension of the unstable, center or stable manifold as stated. In the case (a) it is obvious that the unstable manifold is contained in the plane $\{X=0\}$ since component $X$ is non-increasing along the trajectories in the half-space $\{Y<0\}$ as shown by the equation for $\dot{X}$ in \eqref{PSsyst1}. Concerning the center manifold, we infer from the Local Center Manifold Theorem \cite[Theorem 1, Section 2.10]{Pe} and \cite[Theorem 2.15, Chapter 9]{CH} that any center manifold in a neighborhood of $P_0^{\lambda}$ has to contain an arc of the invariant parabola \eqref{parabola}, hence in the cases (a) and (c) the center manifold is unique and lies locally on the parabola. Finally, the profiles contained on the stable manifolds in the neighborhood of any of the points $P_0^{\lambda}$ enters $P_0^{\lambda}$ with
$$
X=0, \qquad Z=-\lambda^2-\frac{\beta}{\alpha}\lambda, \qquad Y=\lambda<0,
$$
and we obtain from \eqref{PSsyst1} that the profiles $f$ satisfy $f(\xi_0)=0$ at a finite point $\xi_0$ such that
\begin{equation}\label{interm2}
\xi_0^{\sigma-2}=\frac{\alpha^2}{m}\left(-\lambda^2-\frac{\beta}{\alpha}\lambda\right), \qquad f(\xi)\sim\left(K+\frac{\lambda\alpha(m-1)}{2m}\xi^2\right)^{1/(m-1)}, \ {\rm as} \ \xi\to\xi_0.
\end{equation}
for a fixed $K>0$ (depending on $\xi_0$), which is a behavior qualitatively equivalent to \eqref{beh.interf}.
\end{proof}

\noindent \textbf{Remark}. For $\lambda=-\beta/2\alpha$ we obtain the maximum value for $Z$, namely
$$
Z=\frac{\beta^2}{4\alpha^2},
$$
which leads to $\xi=\xi_{max}$. Since the $Z$ component is non-decreasing on the trajectories of the system \eqref{PSsyst1} we get that there is no interface behavior at points $\xi_0>\xi_{max}$.

Although the local analysis near the critical point $P_2$ is totally similar to the one performed in \cite{IS20b}, we will state the result and a sketch of its proof here for the reader's convenience taking into account the importance of $P_2$ for the whole analysis. Moreover, the exact form of the eigenvector tangent to the unique orbit going out of $P_2$ will be used in the sequel.
\begin{lemma}[Local analysis of the point $P_2$]\label{lem.P2}
The system \eqref{PSsyst1} in a neighborhood of the critical point $P_2$ has a two-dimensional stable manifold and a one-dimensional unstable manifold. The stable manifold is contained in the invariant plane $\{Z=0\}$. There exists a unique orbit going out of $P_2$ containing profiles with a local behavior near the origin given in \eqref{beh.P2}
\end{lemma}
\begin{proof}
The linear part of the system \eqref{PSsyst1} near $P_2$ has the matrix
$$
M(P_2)=\frac{1}{2(m+1)\alpha}\left(
  \begin{array}{ccc}
    -2(m-1) & (m-1)^2 & 0 \\
    2(m+1)\alpha-2 & -2\beta(m+1)-(m+3) & -2(m+1)\alpha \\
    0 & 0 & (\sigma-2)(m-1) \\
  \end{array}
\right)
$$
with eigenvalues $\lambda_1$, $\lambda_2$ and $\lambda_3$ such that
$$
\lambda_1+\lambda_2=-\frac{2(m-1)+2(m+1)\beta}{2(m+1)\alpha}<0, \ \lambda_1\lambda_2=\frac{m-1}{2(m+1)\alpha^2}>0
$$
whence $\lambda_1$, $\lambda_2<0$ and
$$
\lambda_3=\frac{(\sigma-2)(m-1)}{2(m+1)\alpha}>0.
$$
It is easy to check (the details are given in \cite[Lemma 3.2]{IS20b}) that the two-dimensional stable manifold is contained in the invariant plane $\{Z=0\}$ and there exists a unique orbit going out of $P_2$ towards the half-space $\{Z>0\}$ tangent to the eigenvector
\begin{equation}\label{interm.bis}
e_3=\left(-\frac{2(m-1)(m+1)\alpha}{(m-1)\sigma^2+(5-m)\sigma+4m},-\frac{2(m+1)\sigma}{(m-1)\sigma^2+(5-m)\sigma+4m},1\right).
\end{equation}
The local behavior \eqref{beh.P2} of the profiles contained in the orbit going out of $P_2$ is obtained from the fact that
$$
X(\xi)=\frac{m}{\alpha}\frac{f^{m-1}(\xi)}{\xi^2}\sim\frac{m-1}{2(m+1)\alpha},
$$
exactly as in \cite[Lemma 3.2]{IS20b}.
\end{proof}
It is interesting to notice that the set of the critical points $P_0^{\lambda}$ with $-\beta/2\alpha<\lambda<0$ can be seen as an "big attractor" jointly. More precisely we have
\begin{proposition}\label{prop.att}
The set of points $\mathcal{S}=\{P_0^{\lambda}: -\beta/2\alpha<\lambda<0\}$ is an asymptotically stable set for the system \eqref{PSsyst1}. Moreover, fixing $\lambda_{m}$ and $\lambda_M$ such that
$$
-\frac{\beta}{2\alpha}<\lambda_m<\lambda_M<0,
$$
then the set $\{P_0^{\lambda}: \lambda_m<\lambda<\lambda_M\}$ is an asymptotically stable set for the system \eqref{PSsyst1}.
\end{proposition}
\begin{proof}
We already know from Lemma \ref{lem.P1} that the one-dimensional center manifold of any of the points in $\mathcal{S}$ is contained in the parabola \eqref{parabola} composed only by critical points, thus the flow on the center manifold satisfies $\dot{Z}=0$ (there is no flow). Since also by Lemma \ref{lem.P1} the points in $\mathcal{S}$ have a two-dimensional stable manifold, we infer from \cite[Theorem 2.1.2, Chapter 2.1]{Wig} that every point $P_0^{\lambda}\in\mathcal{S}$ is a stable point for the system \eqref{PSsyst1}. Let now $\lambda_{m}$ and $\lambda_M$ as in the statement and $\lambda_1$, $\lambda_2$ such that $\lambda_m<\lambda_2<\lambda_1<\lambda_M$. By stability, for any $\lambda\in[\lambda_2,\lambda_1]$ there exist small positive numbers $\delta(\lambda)$, $\epsilon(\lambda)$ such that for any point in a ball $B(P_0^{\lambda},\delta(\lambda))$, the flow of the system \eqref{PSsyst1} starting from this point stays inside the larger ball $B(P_0^{\lambda},\epsilon(\lambda))$. Since the set $\{P_{\lambda}:\lambda_2\leq\lambda\leq\lambda_1\}$ is a compact set, we can extract such a finite covering of it with balls $B(P_0^{\lambda_i},\delta(\lambda_i))$, $i=1,2,...,l$ such that for any point $x$ in the union of these balls, the flow starting at $x$ will stay forever in the union of the balls $B(P_0^{\lambda_i},\epsilon(\lambda_i))$. This gives the desired stability of the whole set. Since the coordinate $X$ is decreasing and the coordinate $Z$ is increasing along the trajectories, there are no limit cycles and all such orbits must enter a critical point lying in the set
$$
\bigcup\limits_{i=1}^lB(P_0^{\lambda_i},\epsilon(\lambda_i)),
$$
that is, one of the points $P_0^{\lambda}$ with $\lambda_m<\lambda<\lambda_M$ which gives the asymptotic stability. In particular, taking the whole arc of parabola we obtain that the whole $\mathcal{S}$ is an asymptotically stable set.
\end{proof}

\medskip

\noindent \textbf{Local analysis of the critical points at infinity}. This is totally identical to the analogous analysis performed for $m+p>2$ in \cite[Section 3]{IS20b} and we will only list the critical points and the behavior of the profiles near them without proofs for the sake of completeness. To study the critical points at infinity we pass to the Poincar\'e hypersphere according to the theory in \cite[Section 3.10]{Pe} and introduce the new variables $(\overline{X},\overline{Y},\overline{Z},W)$ by
$$
X=\frac{\overline{X}}{W}, \ Y=\frac{\overline{Y}}{W}, \ Z=\frac{\overline{Z}}{W}
$$
Using \cite[Theorem 4, Section 3.10]{Pe} we get that the critical points at space infinity lie on the equator of the Poincar\'e hypersphere and with the calculations done in \cite[Section 3]{IS20b} we find the following five critical points (in variables $(\overline{X},\overline{Y},\overline{Z},W)$):
$$
Q_1=(1,0,0,0), \ \ Q_{2,3}=(0,\pm1,0,0), \ \ Q_4=(0,0,1,0), \ \
Q_5=\left(\frac{m}{\sqrt{1+m^2}},\frac{1}{\sqrt{1+m^2}},0,0\right).
$$
We list below the local analysis of each of these points, according to the detailed analysis of them performed in \cite[Section 3]{IS20b} which holds true also when $m+p=2$:

\medskip

$\bullet$ The critical point $Q_1$ on the Poincar\'e hypersphere is an unstable node. The orbits going out of it towards the interior of the phase space associated to the system \eqref{PSsyst1} contain profiles $f$ intersecting the vertical axis at a positive point, that is $f(0)=a>0$ and any possible value of $f'(0)$.

$\bullet$ The critical point $Q_2$ on the Poincar\'e hypersphere is an unstable node. The orbits going out of it towards the interior of the phase space associated to the system \eqref{PSsyst1} contain profiles $f$ with a positive change of sign at some finite $\xi_0\in(0,\infty)$. More precisely, there exists $\xi_0\in(0,\infty)$ such that $f(\xi_0)=0$, $(f^m)'(\xi_0)>0$ and the profile becomes strictly positive in a right-neighborhood of $\xi_0$.

$\bullet$ The critical point $Q_3$ on the Poincar\'e hypersphere is a stable node. The orbits entering it and coming from the interior of the phase space associated to the system \eqref{PSsyst1} contain profiles $f$ with a negative change of sign at some finite $\xi_0\in(0,\infty)$. More precisely, there exists $\xi_0\in(0,\infty)$ with $f(\xi_0)=0$, $(f^m)'(\xi_0)<0$ and the profile is strictly positive in a left-neighborhood of $\xi_0$ and can be extended on the negative side in a right-neighborhood of $\xi_0$.

$\bullet$ The critical point $Q_4$ on the Poincar\'e hypersphere is a non-hyperbolic critical point. Its local analysis is very hard to perform, but also not needed. According to \cite[Lemma 3.6]{IS20b} (whose proof is now very easy as we are only in the case $\xi\to\infty$ and $\sigma>2$) there are no profiles solutions to \eqref{SSODE} contained in the orbits connecting to this critical point.

$\bullet$ The critical point $Q_5$ in the Poincar\'e hypersphere is a hyperbolic critical point having a two-dimensional unstable manifold and a one-dimensional stable manifold. The orbits going out from this point into the finite region of the phase space do it on the unstable manifold and contain the family of profiles with a positive change of sign at $\xi=0$, that is
\begin{equation}\label{beh.Q5}
f(0)=0, \ \quad f(\xi)\sim K\xi^{1/m} \ {\rm as} \  \xi\to0, \ K>0,
\end{equation}
in a right-neighborhood of $\xi=0$.

\medskip

The proofs of all these statements are given with all the details in \cite[Section 3]{IS20b}. We are now ready to pass to the global analysis of the phase space associated to the system \eqref{PSsyst1}, and this is the point where the biggest differences with respect to the analysis in \cite{IS20b} are found.

\section{Existence of good profiles with interface}\label{sec.exist}

This section is devoted to the proof of Theorem \ref{th.exist}. Namely, we will show that for every $\xi_0\in(0,\xi_{max}]$ there exists at least a good profile with interface exactly at the point $\xi=\xi_0$. In the process we will also obtain a family of interesting decreasing supersolutions to Eq. \eqref{eq1}. We stress here that the proof of the analogous result for exponents such that $m+p>2$, performed in \cite[Section 4]{IS20b}, has been done by the technique of backward shooting from the interface point. In our case, due to the lack of uniqueness of profiles with interface at a given $\xi_0\in(0,\xi_{max}]$, this approach is no longer possible and we instead do the job with a deeper global analysis of the geometry of the trajectories entering the critical parabola of the points $P_0^{\lambda}$ in the phase space associated to the system \eqref{PSsyst1}. We begin with the orbits entering the most negative part of the parabola.
\begin{proposition}\label{prop.second}
The orbits entering the critical points $P_0^{\lambda}$ with $\lambda\in(-\beta/\alpha,-\beta/2\alpha)$ contain profiles $f(\xi)$ that are non-increasing and intersect the axis $\xi=0$ with $f(0)=A>0$ and $f'(0)<0$. The self-similar functions
$$
u(x,t)=(T-t)^{-\alpha}f(|x|(T-t)^{\beta}), \qquad T>0, \ t\in(0,T),
$$
with profiles $f$ contained in such orbits are supersolutions to Eq. \eqref{eq1}.
\end{proposition}
\begin{proof}
Let us consider the plane $\{Y=-\beta/2\alpha\}$ in the phase space. The flow of the system over this plane is given by the sign of the expression
$$
F(X,Z)=-Y^2-\frac{\beta}{\alpha}Y+X-XY-Z=\frac{\beta^2}{4\alpha^2}+X\left(1+\frac{\beta}{2\alpha}\right)-Z.
$$
Thus, a trajectory of the system \eqref{PSsyst1} can cross this plane from the right to the left only in the region where $F(X,Z)\leq0$, that is,
\begin{equation}\label{interm3}
Z\geq\frac{\beta^2}{4\alpha^2}+X\left(1+\frac{\beta}{2\alpha}\right)>\frac{\beta^2}{4\alpha^2}.
\end{equation}
Since the component $Z$ is non-decreasing along the trajectories of the system \eqref{PSsyst1}, it cannot go down, thus the inequality \eqref{interm3} stays true along the trajectory in the region $\{Y<-\beta/2\alpha\}$. It thus follows that no connection which crosses the plane $\{Y=-\beta/2\alpha\}$ can enter any of the points $P_0^{\lambda}$ with $\lambda\in(-\beta/\alpha,-\beta/2\alpha)$ since the highest value of the component $Z$ on the critical parabola \eqref{parabola} is $Z=\beta^2/4\alpha^2$, attained for $\lambda=-\beta/2\alpha$. This means that the orbits entering any point $P_0^{\lambda}$ with $\lambda\in(-\beta/\alpha,-\beta/2\alpha)$ on the (one-dimensional) stable manifold of it have to lie completely in the half-space $\{Y<-\beta/2\alpha\}$. Thus the corresponding profiles are strictly decreasing before reaching their interface and have to come from the critical point $Q_1$ at infinity, thus intersecting the vertical axis $\xi=0$ with a negative slope $f'(0)<0$, as stated. Finally, the self-similar functions $u(x,t)$ whose profile is of this type are solutions to Eq. \eqref{eq1} at $|x|>0$ and supersolutions at the origin.
\end{proof}

\noindent \textbf{Remarks.} (a) The proof above also shows that any orbit of the system crossing the plane $\{Y=-\beta/2\alpha\}$ has to enter the critical point $Q_3$ at infinity, thus the profiles contained in it will have a negative change of sign at some $\xi_0\in(0,\infty)$ but no interface.

(b) The decreasing supersolutions in self-similar form given by Proposition \ref{prop.second} will be strongly used for comparison in the companion paper \cite{IMS20} where the qualitative analysis of a similar equation to Eq. \eqref{eq1}, more precisely
\begin{equation}\label{eq.ims}
u_t=(u^m)_{xx}+(1+|x|)^{\sigma}u^p, \qquad 0<p<1, \ m>1,
\end{equation}
is performed. They will be helpful in the proof of the finite speed of propagation of solutions to \eqref{eq.ims} for compactly supported initial data when $m+p=2$.

\medskip

The global description of the trajectories of the system \eqref{PSsyst1} entering the critical points $P_0^{\lambda}$ with $\lambda\in[-\beta/2\alpha,0)$ is more involved. Let us consider the parabolic cylinder
\begin{equation}\label{cyl}
-Y^2-\frac{\beta}{\alpha}Y-Z=0.
\end{equation}
We have the following technical result.
\begin{lemma}\label{lem.inner}
The orbits in the phase space associated to the system \eqref{PSsyst1} entering one of the points $P_0^{\lambda}$ with $\lambda\in[-\beta/2\alpha,0)$ from the interior of the parabolic cylinder \eqref{cyl} stay forever in the half-space $\{Y<0\}$. The same holds true for those orbits entering one of the points
$P_0^{\lambda}$ with $\lambda\in[-\beta/2\alpha,0)$ from the exterior of the cylinder \eqref{cyl} but which on their trajectory have previously crossed the parabolic cylinder at a point lying in the half-space $\{Z>X\}$. The profiles contained in these orbits are non-increasing.
\end{lemma}
\begin{proof}
We divide the proof into several steps for the reader's convenience.

\medskip

\noindent \textbf{Step 1.} The normal direction to the parabolic cylinder \eqref{cyl} is given by $\overline{n}=(0,-2Y-\beta/\alpha,-1)$, thus the direction of the flow of the system \eqref{PSsyst1} on the cylinder \eqref{cyl} depends on the sign of the expression
\begin{equation*}
\begin{split}
F(X,Y)&=\left(-2Y-\frac{\beta}{\alpha}\right)X(1-Y)+(\sigma-2)X\left(Y^2+\frac{\beta}{\alpha}Y\right)\\
&=X\left[(\sigma-2)Y^2+(\sigma-2)\frac{\beta}{\alpha}Y-2Y-\frac{\beta}{\alpha}+2Y^2+\frac{\beta}{\alpha}Y\right]\\
&=X\left[\sigma Y^2+\left((\sigma-1)\frac{\beta}{\alpha}-2\right)Y-\frac{\beta}{\alpha}\right]=Xh(Y),
\end{split}
\end{equation*}
where
$$
h(Y)=Y\left[\sigma Y+(\sigma-1)\frac{\beta}{\alpha}\right]-\left(2Y+\frac{\beta}{\alpha}\right).
$$
For $Y\in[-\beta/2\alpha,0]$ we notice that
$$
\sigma Y+(\sigma-1)\frac{\beta}{\alpha}\geq-\frac{\beta}{2\alpha}\sigma+(\sigma-1)\frac{\beta}{\alpha}=\frac{\beta}{2\alpha}(\sigma-2)>0,
$$
thus, taking into account that $Y\in[-\beta/2\alpha,0]$ we readily get that $h(Y)\leq0$. It follows that on this side of the cylinder \eqref{cyl} (the half of it where $Y\in[-\beta/2\alpha,0]$) the direction of the flow is only pointing from inside the cylinder towards outside.

\medskip

\noindent \textbf{Step 2.} It is obvious now that a trajectory coming from the half-space $\{Y>0\}$, cannot enter the interior of the cylinder \eqref{cyl}, since by Step 1 it cannot cross the cylinder with $Y\in[-\beta/2\alpha,0]$ and this forces the orbit to cross the plane $\{Y=-\beta/2\alpha\}$ and thus enter $Q_3$ according to Remark (a) after Proposition \ref{prop.second}. Thus a trajectory entering any of the points $P_0^{\lambda}$ from the interior of the cylinder \eqref{cyl} must live forever in the region $\{Y<0\}$.

\medskip

\noindent \textbf{Step 3.} Let us finally consider the case of an orbit entering $P_0^{\lambda}$ with $\lambda\in[-\beta/2\alpha,0)$ from the outside of the cylinder but after having crossed the cylinder \eqref{cyl} at a point in the region $\{Z>X\}$. We show that this orbit cannot cross the plane $\{Y=0\}$. Let $\eta_0$ be the independent variable such that $(X(\eta_0),Y(\eta_0),Z(\eta_0))$ is the intersection point of the orbit with the cylinder \eqref{cyl}, where $Z(\eta_0)>X(\eta_0)$ as in the statement. The orbit cannot have crossed $\{Y=0\}$ at some point with $\eta<\eta_0$ as proved in Step 2. For $\eta>\eta_0$ the orbit is already outside the cylinder. On the one hand, the direction of the flow on the plane $\{X-Z=0\}$ is given by the sign of the expression
$$
X((m-1)Y-2X)-(\sigma-2)XZ=X((m-1)Y-\sigma X)<0,
$$
when $Y<0$, hence it follows that the plane cannot be crossed in the half-space $\{Y<0\}$ and thus $X(\eta)<Z(\eta)$ for any $\eta>\eta_0$ while the orbit lies in $\{Y<0\}$. On the other hand, the direction of the flow on the plane $\{Y=0\}$ is given by the sign of the difference $X-Z$, thus this plane can be crossed from the left to the right only in the region $\{X>Z\}$, which our orbit will never reach according to the previous calculation. We thus get that the orbit under consideration will stay forever in the half-space $\{Y<0\}$ and the profiles contained in it will be non-increasing, as stated.
\end{proof}
We are now ready to prove a proposition which essentially restates in terms of the phase space the proof of Theorem \ref{th.exist} for the critical points $P_0^{\lambda}$ with $\lambda\in(-\beta/2\alpha,0)$.
\begin{proposition}\label{prop.first}
For any $\lambda\in(-\beta/2\alpha,0)$ there exists at least an orbit entering the critical point $P_0^{\lambda}$ and containing good profiles with interface according to Definition \ref{def1}.
\end{proposition}
\begin{proof}
The main idea of the proof is to classify the connections entering $P_0^{\lambda}$ on the two-dimensional stable manifold of this point as tangent to a one-parameter family of explicit orbits and then perform a shooting with the free parameter of the family ranging between the two limits of the stable manifold: one inside the invariant plane $\{X=0\}$ on the horizontal direction $Z=-\lambda^2-(\beta/\alpha)\lambda$ (coming from the unstable node $Q_2$) and the other on a direction entering through the interior of the cylinder \eqref{cyl} (and coming from the unstable node $Q_1$). The proof is rather technical and will be split into several steps.

\medskip

\noindent \textbf{Step 1.} We translate the point $P_0^{\lambda}$ to the origin by letting
\begin{equation}\label{PSvar3}
Y_1=Y-\lambda, \qquad Z_1=Z+\lambda^2+\frac{\beta}{\alpha}\lambda,
\end{equation}
and obtain the following system
\begin{equation}\label{PSsyst2}
\left\{\begin{array}{ll}\dot{X}=(m-1)\lambda X+(m-1)XY_1-2X^2,\\
\dot{Y}_1=-Y_1^2-\frac{2\alpha\lambda+\beta}{\alpha}Y_1+(1-\lambda)X-XY_1-Z_1,\\
\dot{Z}_1=-\frac{(\sigma-2)\lambda(\alpha\lambda+\beta)}{\alpha}X+(\sigma-2)XZ_1,\end{array}\right.
\end{equation}
where our critical point becomes $(X,Y_1,Z_1)=(0,0,0)$. We notice that the linear approximation of the previous system \eqref{PSsyst2} can be integrated to obtain an approximation of the trajectories in a small neighborhood of the point. We readily get that
$$
\frac{dZ_1}{dX}\sim-\frac{(\sigma-2)(\alpha\lambda+\beta)}{(m-1)\alpha},
$$
that is
\begin{equation}\label{interm4}
Z_1=-\frac{(\sigma-2)(\alpha\lambda+\beta)}{(m-1)\alpha}X+o(|X|),
\end{equation}
where by the notation $o(|X|)$ we understand terms that tend to zero faster than $|X|$. The first order in \eqref{interm4} represents the plane tangent to the stable manifold of the point, generated by the two eigenvectors of the negative eigenvalues. Introducing the approximation given by \eqref{interm4} into the second equation of the system \eqref{PSsyst1} we find the new system (up to order one)
\begin{equation}\label{interm5}
\left\{\begin{array}{ll}\dot{X}=(m-1)\lambda X+O(|(X,Y_1)|^2),\\
\dot{Y}_1=-\frac{2\alpha\lambda+\beta}{\alpha}Y_1+\left(1-\lambda+\frac{(\sigma-2)(\alpha\lambda+\beta)}{(m-1)\alpha}\right)X+O(|(X,Y_1)|^2),\end{array}\right.
\end{equation}

\medskip

\noindent \textbf{Step 2.} We want to integrate in a linear approximation the system \eqref{interm5}. To this end, let us first notice that the linear term in the equation for $\dot{Y}_1$ in \eqref{interm5} cannot vanish. Assume for contradiction that it vanishes in the first order. It thus follows that
\begin{equation}\label{interm6}
Y_1=KX+o(|X|), \qquad K=\frac{\alpha}{2\alpha\lambda+\beta}\left[1-\lambda+\frac{(\sigma-2)(\alpha\lambda+\beta)}{(m-1)\alpha}\right],
\end{equation}
thus $dY_1/dX=K+o(1)$ in a neighborhood of the origin. But in this case, we also infer from the system \eqref{interm5} that
$$
\frac{dY_1}{dX}=\frac{O(|(X,Y_1)|^2}{(m-1)\lambda X}=O(|(X,Y_1)|),
$$
which is a contradiction to \eqref{interm6}. It thus follows that the linear part in the equation for $\dot{Y}_1$ in \eqref{interm5} is nontrivial and we can integrate it up to the first order (linear terms) using always the Hartman-Grobman Theorem to obtain that
\begin{equation}\label{interm7}
Y_1=K_1X^{-\frac{2}{m-1}-\frac{2}{(\sigma+2)\lambda}}-\frac{(\sigma+2)(m-\sigma+1)\lambda-(3\sigma-2)(m-1)}{(m-1)[(\sigma+2)(m+1)\lambda+2(m-1)]}X,
\end{equation}
with $K_1\in\real$ arbitrary. We deduce that every trajectory entering $(X,Y_1,Z_1)=(0,0,0)$ on the two-dimensional stable manifold does that tangent to a trajectory as in \eqref{interm7} with some $K_1\in\real$. We can thus identify these trajectories with the parameter $K_1$ as above. Let us also notice at this point that
$$
-\frac{2}{m-1}-\frac{2}{(\sigma+2)\lambda}>0 \ \ {\rm if \ and \ only \ if} \ \ \lambda\in\left(-\frac{\beta}{2\alpha},0\right),
$$
a fact that confirms our previous analysis leading to only one trajectory entering $P_0^{\lambda}$ for $\lambda\in(-\beta/\alpha,-\beta/2\alpha)$. Indeed, in that case the power of $X$ in the first term of \eqref{interm7} becomes positive and in order to be tangent to the unstable manifold we are obliged to put $K_1=0$ to vanish that term. In our range of $\lambda\in(-\beta/2\alpha,0)$ we thus find a one-parameter family of trajectories entering $P_0^{\lambda}$.

\medskip

\noindent \textbf{Step 3. Trajectories contained in $\{Y<0\}$}. Let us take a ball around the point $P_0^{\lambda}$ (understood as a ball around the origin in variables $(X,Y_1,Z_1)$) with $X<\epsilon$ and $-\epsilon<Z_1<0$ for some $\epsilon$ sufficiently small. More precisely, we let $\epsilon$ such that
$$
0<\epsilon<\frac{1}{2}\left[-\lambda^2-\frac{\beta}{\alpha}\lambda\right]
$$
There exists then $\overline{K}_1$ depending on $\epsilon$ such that for any $K_1\in(-\infty,\overline{K}_1)$ we have $Y_1<0$ on the boundary of that neighborhood at the points where $X=\epsilon$. It thus follows that these orbits pass through the interior of the cylinder \eqref{cyl}, in particular at this point with $X=\epsilon$ and $Z_1<0$ (this follows from the fact that $Z$ is increasing along the trajectories, thus $Z_1$ too). Since all these orbits enter $P_0^{\lambda}$, they do that either through the interior of the cylinder \eqref{cyl} or crossing the cylinder at a later point along the trajectory. But since we are in the half-space $\{Y<0\}$, the component $X$ decreases along the trajectories and the component $Z$ increases along the trajectories, hence at the later point where the trajectory crosses the cylinder we have
$$
X<\epsilon, \qquad Z>-\lambda^2-\frac{\beta}{\alpha}\lambda-\epsilon>\epsilon,
$$
which implies that $Z>X$ at the point where the trajectory crosses the cylinder \eqref{cyl}. We infer from Lemma \ref{lem.inner} that in any of the two possible cases the orbits are fully contained in the half-space $\{Y<0\}$ (and in particular, taking into account the local analysis of the phase space, have to go out of the point $Q_1$).

\medskip

\noindent \textbf{Step 4. Trajectories coming from the critical point $Q_2$}. Considering the invariant plane $\{X=0\}$, the system \eqref{PSsyst1} reduces to the following equations
$$
\dot{Y}=-Y^2-\frac{\beta}{\alpha}Y-Z, \qquad \dot{Z}=0,
$$
thus it is easy to see that there exists a connection inside the plane $\{X=0\}$ entering the point $P_0^{\lambda}$ and going on the half-line with constant component $Z$, namely
\begin{equation}\label{interm8}
\{X=0, Y>\lambda, Z=-\lambda^2-\frac{\beta}{\alpha}\lambda\}.
\end{equation}
Let us take for some $\epsilon>0$ small the point
$$
P_{\epsilon}^{\lambda}=\left(0,\lambda+\epsilon,-\lambda^2-\frac{\beta}{\alpha}\lambda\right)
$$
lying on the half-line \eqref{interm8} and consider the ball $B(P_{\epsilon}^{\lambda},\epsilon/2)$ in the phase space associated to the system \eqref{PSsyst1}. Taking some $y_0>1$ (to be determined later), we infer by the continuous dependence theorem \cite[Theorem 1, Section 2.3]{Pe} that there exists $\delta>0$ sufficiently small (depending on $\epsilon$) such that all the orbits entering the ball $B(P_{\epsilon}^{\lambda},\epsilon/2)$ pass through the ball $B(R,\delta)$, where $R$ is the point in the half-line \eqref{interm8} with $Y=y_0$, that is
$$
R=\left(0,y_0,-\lambda^2-\frac{\beta}{\alpha}\lambda\right).
$$
It is easy to see that, by letting $\epsilon>0$ eventually smaller, the stable manifold $W_s$ of the point $P_0^{\lambda}$ intersects the ball $B(P_{\epsilon}^{\lambda},\epsilon/2)$. Let $K_{1,0}$ be the parameter (corresponding to \eqref{interm7}) of a connection entering $P_0^{\lambda}$ after intersecting the ball $B(P_{\epsilon}^{\lambda},\epsilon/2)$ and let us take another parameter $K_1>K_{1,0}$ and consider the corresponding connection according to \eqref{interm7}. Fixing the plane $\{Y=\lambda+\epsilon\}$, it readily follows from \eqref{interm7} that the connection intersects this plane at a coordinate $X$ smaller than the one corresponding to the intersection of the connection with parameter $K_{1,0}$. The contrary happens to the component $Z_1$ since $Z_1$ is negative but has a linear dependence on $X$, thus the component $Z_1$ corresponding to the parameter $K_1$ when crossing the plane $\{Y=\lambda+\epsilon\}$ is bigger (that is, closer to zero) than the one corresponding to the one with parameter $K_{1,0}$. All these considerations prove that any connection from the one-parameter family \eqref{interm7} with parameter $K_1\in(K_{1,0},\infty)$ intersects the ball $B(P_{\epsilon}^{\lambda},\epsilon/2)$ and thus these orbits also pass through the ball $B(R,\delta)$. To end this step, we analyze the plane $\{Y=y_0\}$: the direction of the flow of the system \eqref{PSsyst1} on it is given by the sign of the expression
$$
H(X,Z)=-y_0^2-\frac{\beta}{\alpha}y_0+X(1-y_0)-Z,
$$
which is strictly negative provided $y_0>1$. Thus the plane $\{Y=y_0\}$ can be crossed only from the right to the left. By choosing $y_0>1$ in the previous arguments, we infer that the orbits entering $P_0^{\lambda}$ and with parameter $K_1\in(K_{1,0},\infty)$ in \eqref{interm7} must come out of a critical point with component $Y>y_0-\delta$. Since in the previous arguments $y_0$ can be taken finite but as large as we wish (changing $\delta$ accordingly) we find that there are orbits coming out of $Q_2$ and entering $P_0^{\lambda}$.

\medskip

\noindent \textbf{Step 5. Good profiles entering $P_0^{\lambda}$}. According to Step 3, let us take $\overline{K}\in\real$ to be the supremum of the parameters $K_1$ of all orbits in the one-parameter family \eqref{interm7} which go out in the region $\{Y<0\}$. These orbits come from the critical point $Q_1$ at infinity. Since $Q_1$ is an unstable node, the set of parameters $K_1$ such that the corresponding orbit goes out in the region $\{Y<0\}$ is an open set, thus $\overline{K}$ does not belong to this set. Moreover, we infer from Step 4 and the fact that $Q_2$ is also an unstable node that the set of parameters $K_1$ as in \eqref{interm7} corresponding to orbits going out of $Q_2$ is also an open set, hence $\overline{K}$ does also not belong to this set. It thus remains for the orbit corresponding to the parameter $\overline{K}$ in \eqref{interm7} to come out from one of the points $Q_1$ (but in that case not with slope $Y<0$), $P_0$, $P_2$ or $Q_5$. We can remove $Q_5$ from the list by noticing that, by definition, all the orbits with parameter $K_1<\overline{K}$ in \eqref{interm7} go out from the critical point $Q_1$ with negative slope, that is for any profile $f(\xi)$ contained in them we have
$$
(f^m)'(0)=mf^{m-1}(0)f'(0)<0,
$$
hence in the limit case corresponding to parameter $\overline{K}$ we have profiles with $(f^m)'(0)\leq0$, while the profiles contained in orbits going out of $Q_5$ do that with $(f^m)'(0)=K>0$ as it follows from \eqref{beh.Q5}. On the other hand, it is easy to find that, if the orbit comes out of $Q_1$, it cannot do so with positive slope by the same reason. It thus follows that the orbit with parameter $\overline{K}$ in \eqref{interm7} might come out from $P_0$, from $P_2$ or from $Q_1$ but in the latter case containing profiles with $f'(0)=0$. In all these cases, we discovered an orbit entering $P_0^{\lambda}$ and containing good profiles with interface, as claimed.
\end{proof}
We finally deal with the critical point $P_0^{\lambda}$ with $\lambda=-\beta/2\alpha$, which is the maximum point of the critical parabola \eqref{parabola}. Its analysis is technically more involved, as we shall see below.
\begin{proposition}\label{prop.max}
There exists at least an orbit entering the critical point $P_0^{\lambda}$ with $\lambda=-\beta/2\alpha$ and containing good profiles with interface according to Definition \ref{def1}.
\end{proposition}
\begin{proof}
This proof is based on the same idea as the previous one, by classifying the trajectories entering $P_0^{\lambda}$ through the tangency to an explicit one-parameter family of orbits. This is much more involved for $\lambda=-\beta/2\alpha$ since we deal with a \emph{very interesting case of a center-stable manifold} joining a one-dimensional stable manifold and a one-parameter family of center manifolds near $P_0^{\lambda}$ which are \emph{not analytic} and have an exponential form that will be made precise below. We divide as usual the proof into several steps for easiness of the reading.

\medskip

\noindent \textbf{Step 1.} In a first step, we translate the point $P_0^{\lambda}$ with $\lambda=-\beta/2\alpha$ to the origin of the space and then put it into a \emph{normal form}. Let us first pass again to the new variables $(X_1,Y_1,Z_1)$ as in \eqref{PSvar3} (where $X_1=X$ to ease the notation) and to the corresponding autonomous system \eqref{PSsyst2}, taking into account the noticeable difference that the linear term in $Y_1$ in the equation for $\dot{Y}_1$ in \eqref{PSsyst2} vanishes. We further perform the following change of variables
\begin{equation}\label{PSvar4}	
X_2=X_1, \qquad Y_2 =CX_1+DY_1, \qquad Z_2 = AX_1+BZ_1,
\end{equation}
where
$$
A=\frac{(\sigma-2)\beta^2}{\alpha^2}, \ B=\frac{2\beta(m-1)}{\alpha}, \ C=\frac{2\alpha(m-1)+\beta(m+\sigma-3)}{m-1}, \ D = (m-1)\beta,
$$
to obtain the following system
\begin{equation}\label{PSsyst3}
\left\{\begin{array}{ll}\dot{X}_2=-\frac{(m-1)\beta}{2\alpha}X_2-\frac{(2\alpha+3\beta)(m-1)+\beta(\sigma-2)}{\beta(m-1)}X_2^2+\frac{1}{\beta}X_2Y_2,\\
\dot{Y}_2=-\frac{\alpha}{2}Z_2-\frac{1}{(m-1)\beta}Y_2^2+EX_2Y_2-FX_2^2,\\
\dot{Z}_2=(\sigma-2)\left[\frac{\beta}{\alpha^2}X_2Y_2+X_2Z_2-\frac{\beta(\beta m\sigma+2\alpha(m-1)+\beta m-3\beta)}{(m-1)\alpha^2}X_2^2\right],\end{array}\right.
\end{equation}
where
\begin{equation*}
\begin{split}
&E=\frac{2\alpha m^2+\beta\sigma(m+1)-2\alpha-4\beta}{(m-1)^2\beta}, \\ &F=\frac{[m\beta(\sigma-2)+(2m\beta+2m\alpha-\beta)(m-1)][(2\alpha+\beta)(m-1)+\beta(\sigma-2)]}{(m-1)^3\beta}.
\end{split}
\end{equation*}

\medskip

\noindent \textbf{Step 2.} We readily infer from the first and third equations of the system \eqref{PSsyst2} that in a sufficiently small neighborhood of the point $(X_1,Y_1,Z_1)=(0,0,0)$ we have
$$
\frac{dZ_1}{dX_1}\sim-\frac{(\sigma-2)\lambda(\alpha\lambda+\beta}{\alpha(m-1)\lambda}=-\frac{\sigma-2}{\sigma+2},
$$
whence we can write $Z_1=-(\sigma-2)X_1/(\sigma+2)+o(X_1)$. Taking into account the precise values of $A$ and $B$ in \eqref{PSvar4} and the fact that $X_2=X_1$, $Z_2=AX_1+BZ_1$, we deduce very easily from the above first order approximation that in a sufficiently small neighborhood at the origin of the system \eqref{PSsyst3} we have $Z_2=o(X_2)$. We now pass to the study of the center manifolds (that in this case will not be unique) near the origin for the system \eqref{PSsyst3} having the form $X_2=h(Y_2,Z_2)$ for suitable functions $h$. At a formal level, neglecting higher order terms and also neglecting $Z_2$ (which is of lower order with respect to $X_2$ as proved and we expect to be of also lower order with respect to $Y_2^2$) we are left with the reduced system
$$
\dot{X}_2=-\frac{(m-1)\beta}{2\alpha}X_2, \qquad \dot{Y}_2=-\frac{1}{(m-1)\beta}Y_2^2
$$
which by integration leads to the one-parameter family of solutions
\begin{equation}\label{center.man}
X_2=K\exp\left(-\frac{(m-1)^2\beta^2}{2\alpha Y_2}\right)=:Kh(Y_2), \qquad K\geq0.
\end{equation}
We show that indeed the one-parameter family in \eqref{center.man} are the first order approximations to a one-parameter family of non-analytic center manifolds near the origin for the system \eqref{PSsyst3}. Indeed, a center manifold is given by the generic expression $X_2=h(Y_2,Z_2)$ and satisfies a rather complicated partial differential equation according to the Local Center Manifold Theorem, see for example \cite[Section 2.5]{Carr} or \cite[Theorem 1, Section 2.12]{Pe} . We plug in the equation of the center manifold the function $Kh(Y_2)$ as in \eqref{center.man} (for any $K\geq0$). After some calculations, the left-hand side of the equation of the center manifold becomes
\begin{equation*}
\begin{split}
N(Y_2,Z_2)&=\frac{K(m-1)^2\beta^2}{2\alpha Y_2^2}\exp\left(-\frac{(m-1)^2\beta^2}{2\alpha Y_2}\right)\left[-\frac{\alpha}{2}Z_2+EKY_2\exp\left(-\frac{(m-1)^2\beta^2}{2\alpha Y_2}\right)\right.\\
&\left.-FK^2\exp\left(-\frac{(m-1)^2\beta^2}{\alpha Y_2}\right)\right]
+\frac{K}{\beta}Y_2\exp\left(-\frac{(m-1)^2\beta^2}{2\alpha Y_2}\right)\\&-K^2\frac{(2\alpha+3\beta)(m-1)+\beta(\sigma-2)}{(m-1)\beta}\exp\left(-\frac{(m-1)^2\beta^2}{\alpha Y_2}\right)\\
&=o\left(\exp\left(-\frac{(m-1)^2\beta^2}{2\alpha Y_2}\right)\right),
\end{split}
\end{equation*}
since as we have proved $Z_2=o(X_2)$, thus it is of lower order term than the exponential in $h(Y_2)$. The coefficients $E$ and $F$ above are the ones of the system \eqref{PSsyst3}. We thus obtain a one-parameter family of center manifolds in a neighborhood of the point $(X_2,Y_2,Z_2)=(0,0,0)$ having the form
\begin{equation}\label{center.man2}
X_2=K\exp\left(-\frac{(m-1)^2\beta^2}{2\alpha Y_2}\right)+o\left(\exp\left(-\frac{(m-1)^2\beta^2}{2\alpha Y_2}\right)\right), \ K\geq0.
\end{equation}
There is an orbit entering the origin of the system \eqref{PSsyst3} tangent to any of these center manifolds, which allows us to perform a shooting method with respect to the parameter $K$.

\medskip

\noindent \textbf{Step 3. The shooting.} Since we are interested in the behavior only in a sufficiently small neighborhood of the origin in the system \eqref{PSsyst3}, we let $\epsilon>0$ sufficiently small such that if $0<X_2<\epsilon$, $0<Y_2<\epsilon$ we have $\dot{X}_2<0$ in the first equation of the system \eqref{PSsyst3}. Indeed, this is easily done by noticing that the dominating linear term in that equation has a negative coefficient. We infer that if an orbit enters the region $0<X_2<\epsilon$, it will remain there as the $X_2$ component will decrease to zero along the orbit and then the same happens to the component $Y_2$ accordingly to \eqref{center.man}. Moreover, given any point $(X_2,Y_2)$ such that $X_2\in(0,\epsilon)$, $Y_2\in(0,\epsilon)$ and $X_2^2+Y_2^2=\epsilon^2$, there exists an orbit on the family of center manifolds passing through that point on the projection on the $(X_2,Y_2)$-plane, namely the orbit with the value of $K$ given by
\begin{equation}\label{interm20}
K=X_2\exp\left(\frac{(m-1)^2\beta^2}{2\alpha Y_2}\right).
\end{equation}
It is easy to see that the limits of this family of center manifolds are given on one side by the plane $\{X_2=0\}$, corresponding to $K=0$, and on the other side by the plane $\{Y_2=0\}$, corresponding to the limit as $K\to\infty$. Coming back to the initial variables $(X,Y,Z)$, since $X_2=X$ and inside the invariant plane $\{X=0\}$ we have one orbit entering the critical point $P_0^{\lambda}$ with $\lambda=-\beta/2\alpha$ through the horizontal line $Z=\beta^2/4\alpha^2$ coming from the unstable node $Q_2$ at infinity, on the one hand we can just repeat Step 4 of the proof of Proposition \ref{prop.first} to obtain that there exists $K_1>0$ such that for any $K\in(0,K_1)$ the orbits entering tangent to the center manifolds \eqref{center.man} come from $Q_2$. On the other hand, the plane $\{Y_2=0\}$ becomes in the initial variables of the system \eqref{PSsyst1} the plane of equation
$$
CX+D\left(Y+\frac{\beta}{2\alpha}\right)=0, \ \ C=\frac{2\alpha(m-1)+\beta(m+\sigma-3)}{m-1}, \ \ D = (m-1)\beta,
$$
that means
$$
Y=-\frac{\beta}{2\alpha}-\frac{C}{D}X<-\frac{\beta}{2\alpha}.
$$
Since for $K$ large enough the orbits tangent to center manifolds as in \eqref{interm20} approach the plane $\{Y_2=0\}$ we infer that there is $K_2>0$ such that the orbits corresponding to any $K\in(K_2,\infty)$ enter the critical point $P_0^{\lambda}$ with $\lambda=-\beta/2\alpha$ after passing by points with coordinates $Y<-\beta/2\alpha$ lying inside the parabolic cylinder \eqref{cyl} provided $\epsilon$ is small enough. Thus, such orbit either enters the critical point $P_0^{\lambda}$ with $\lambda=-\beta/2\alpha$ through the interior of the cylinder \eqref{cyl}, or it first crosses the boundary of the cylinder in the region $\{Y>-\beta/2\alpha\}$, but in such case at the crossing point we have $X=X_2<\epsilon<Z$ (since $Z$ is close enough to $\beta^2/4\alpha^2$). In both cases Lemma \ref{lem.inner} gives that such orbits live forever in the region $\{Y<0\}$ and come from the unstable node $Q_1$ at infinity. We can then completely repeat Step 5 in the proof of Proposition \ref{prop.first} to reach the conclusion.
\end{proof}

We gather in Figure \ref{fig1} plots of the local manifolds in a neighborhood of all the three types of points on the critical parabola (the first half of it, the vertex and the second half of it) as proved in Propositions \ref{prop.second}, \ref{prop.first} and \ref{prop.max}.

\begin{figure}[ht!]
  \begin{center}
  \includegraphics[width=11cm,height=8cm]{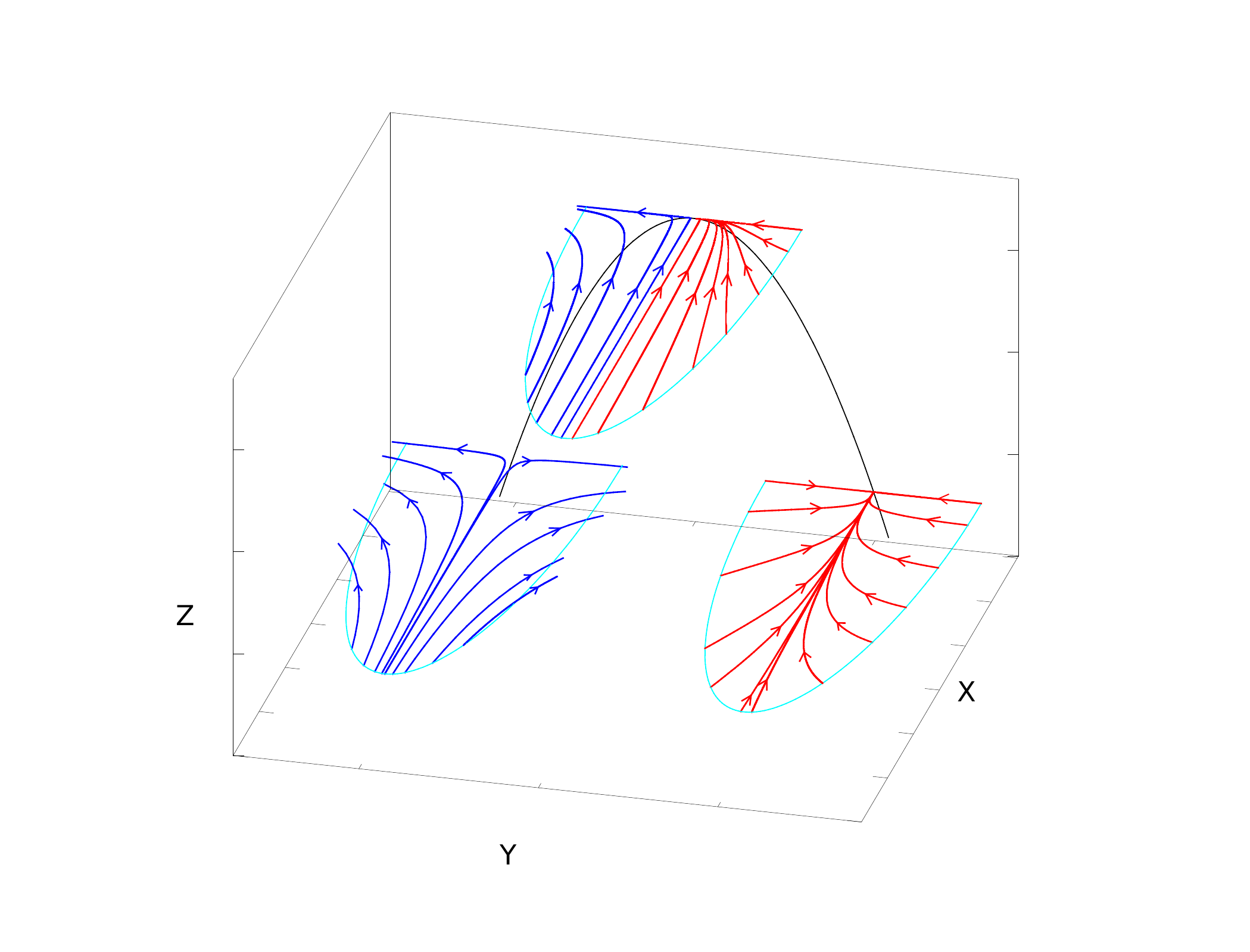}
  \end{center}
  \caption{The local manifolds in a neighborhood of the critical points $P_0^{\lambda}$}\label{fig1}
\end{figure}

\noindent \textbf{Remark.} Let us notice here that the behavior of the orbits in Proposition \ref{prop.max} and its proof give a very interesting example of a \emph{center-stable manifold} near a critical point. This example is remarkable since, with respect to standard theory (see for example \cite{Ke67}) we do not compose the center-stable manifold with all the center manifolds of $P_0^{\lambda}$ with $\lambda=-\beta/2\alpha$ but instead we neglect one part of them (the one lying on the critical parabola \eqref{parabola}) and we generate the center-stable manifold between the remaining part of the center manifolds ranging from the one lying inside the invariant plane $\{X=0\}$ (the horizontal line $\{Z=\beta^2/4\alpha^2\}$) to the trajectory tangent to the eigenvector of the unique negative eigenvalue of the linearized system near $P_0^{\lambda}$.

We are now in a position to end the proof of Theorem \ref{th.exist}.
\begin{proof}[Proof of Theorem \ref{th.exist}]
This is now an immediate consequence of Propositions \ref{prop.first} and \ref{prop.max} and of the fact that the interface points $\xi_0\in(0,\xi_{max}]$ are in bijection with the points of the first half of the parabola \eqref{parabola} through the definition of the variable $Z$ in \eqref{PSvar1}, namely
$$
\xi_0=\left(\frac{\alpha^2}{m}Z\right)^{1/(\sigma-2)}\in(0,\xi_{max}],
$$
for any $Z=-\lambda^2-\frac{\beta}{\alpha}\lambda$ and $\lambda\in[-\beta/2\alpha,0)$.
\end{proof}

\section{Classification of the orbits}\label{sec.class}

In this section we classify the orbits containing good profiles with interface (entering points $P_0^{\lambda}$ with $\lambda\in[-\beta/2\alpha,0)$) according to the critical point from where they begin, which defines the behavior of the profiles at $\xi=0$. Some of the proofs are highly technical and part of the calculations in them were performed with the assistance of a symbolic calculation software. Before passing to the main statements, let us recall the following very useful result.
\begin{lemma}\label{lem.monot}
The component $X$ is decreasing and the component $Y$ is also decreasing in the half-space $\{Y\geq0\}$ along the trajectory going out of the point $P_2$.
\end{lemma}
The proof of Lemma \ref{lem.monot} is given in \cite[Lemma 6.1]{IS20b}, since a simple inspection of the proof shows that it holds true also when $m+p=2$ provided $\sigma>2$. The core of the classification is controlling the unique orbit going out of the critical point $P_2$. This orbit has a different behavior according to the magnitude of $\sigma>2$. We begin with its behavior for $\sigma$ sufficiently small.
\begin{proposition}\label{prop.small}
There exists $\sigma_0>2$ such that for any $\sigma\in(2,\sigma_0)$ the only orbit going out of $P_2$ enters one of the points $P_0^{\lambda}$ with $\lambda\in[-\beta/2\alpha,0)$. Moreover, for these values of $\sigma$ all the orbits going out of $P_0$ also enter some of the points $P_0^{\lambda}$ with $\lambda\in[-\beta/2\alpha,0)$.
\end{proposition}
\begin{proof}
The proof is technical and based on limiting the orbits by barriers in the phase-space. More precisely, we build a region of the plane limited by a number of well-chosen planes together with the boundary of the parabolic cylinder \eqref{cyl} which cannot be left by an orbit from inside, and show that for $\sigma$ sufficiently close to 2 the orbit going out of $P_2$ enters this region. We divide it into several steps for the reader's convenience. Let us recall here that we have $1<m<2$, that is $0<m-1<1$, a bound that will be strongly used throughout this proof.

\medskip

\noindent \textbf{Step 1.} Let us consider as a first barrier the following plane
\begin{equation}\label{plane1}
cY+Z=d, \qquad c=\frac{(m-1)^2}{(\sigma+2)^2}, \ d=\frac{(m-1)^2}{2(\sigma+2)^2}.
\end{equation}
The normal direction to the plane is given by the vector $(0,c,1)$ and the flow of the system \eqref{PSsyst1} on this plane is given by the sign of the expression
\begin{equation*}
\begin{split}
H(X,Y)=&-\frac{(m-1)^2}{(\sigma+2)^2}Y^2-\frac{(m-1)^2)(\sigma-1)}{(\sigma+2)^2}XY+\frac{\sigma(m-1)^2}{2(\sigma+2)^2}X\\
&-\frac{(2\sigma+5-m)(m-1)^3}{(\sigma+2)^4}Y-\frac{(m-1)^4}{2(\sigma+2)^4}.
\end{split}
\end{equation*}
We want to have $H(X,Y)<0$. This is in particular satisfied when the following two conditions on $X$ and $Y$ hold true:
\begin{equation}\label{interm9}
Y>Y^*:=-\frac{m-1}{6(2\sigma+5-m)}, \qquad 0<X<X^*:=\frac{(m-1)^2}{3\sigma(\sigma+2)^2}.
\end{equation}
Indeed, assume that the bounds in \eqref{interm9} are in force. Then it is immediate to check that, on the one hand, the definition of $Y^*$ insures
\begin{equation}\label{interm10}
-\frac{(2\sigma+5-m)(m-1)^3}{(\sigma+2)^4}Y-\frac{(m-1)^4}{6(\sigma+2)^4}<0
\end{equation}
and on the other hand the definition of $X^*$ gives
\begin{equation}\label{interm11}
\frac{\sigma(m-1)^2}{2(\sigma+2)^2}X-\frac{(m-1)^4}{6(\sigma+2)^4}<0.
\end{equation}
Finally, we obtain the following bound for the product $XY$
$$
XY>X^*Y^*=-\frac{(m-1)^3}{18\sigma(2\sigma+5-m)(\sigma+2)^2}>-\frac{(m-1)^2}{6(\sigma-1)(\sigma+2)^2},
$$
whence
\begin{equation}\label{interm12}
-\frac{(m-1)^2)(\sigma-1)}{(\sigma+2)^2}XY-\frac{(m-1)^4}{6(\sigma+2)^4}<0.
\end{equation}
Summing up the inequalities in \eqref{interm10}, \eqref{interm11} and \eqref{interm12} we obtain that $H(X,Y)<0$ provided that the bounds in \eqref{interm9} hold true. The intersection between the plane \eqref{plane1} and the parabolic cylinder \eqref{cyl} is composed by two straight lines with constant values of $Y$ and $Z$ given by the two solutions of the equation
$$
P(Y)=Y^2+\left(\frac{2(m-1)}{\sigma+2}-\frac{(m-1)^2}{(\sigma+2)^2}\right)Y+\frac{(m-1)^2}{2(\sigma+2)^2}=0,
$$
where it is easy to check that the discriminant is positive. However, we are only interested in the range limited by \eqref{interm9}, thus let us remark that
$$
P(Y^*)=\frac{(m-1)^2(49\sigma^2-48m\sigma+244\sigma+12m^2-120m+304)}{36(2\sigma+5-m)^2(\sigma+2)^2}>0
$$
and
$$
P'(Y^*)=\frac{m-1}{\sigma+2}\left[2-\frac{\sigma+2}{6\sigma+3(5-m)}-\frac{m-1}{\sigma+2}\right]>0,
$$
where for the last inequalities we strongly made use of the fact that $\sigma+2>4$ and $m-1<1$. It thus follows that there is no intersection between the plane \eqref{plane1} and the parabolic cylinder \eqref{cyl} in the region $\{Y>Y^*\}$.

\medskip

\noindent \textbf{Step 2.} Consider now the plane
\begin{equation}\label{plane2}
aX+Z=b, \qquad a=b=\frac{(m-1)^2(3\sigma+7-m)}{3(\sigma+2)^2(2\sigma+5-m)}.
\end{equation}
The direction of the flow of the system \eqref{PSsyst1} on the plane \eqref{plane2} is given by the sign of the expression
$$
L(X,Y)=\frac{(m-1)^2(3\sigma+7-m)}{3(\sigma+2)^2(2\sigma+5-m)}X\left[-\sigma X+(m-1)Y+\sigma-2\right]
$$
and it is negative if $-\sigma X+(m-1)Y+\sigma-2<0$. Consider now the straight line $r_1$ inside the plane \eqref{plane2} whose projection on the plane $\{Z=0\}$ has the equation $-\sigma X+(m-1)Y+\sigma-2=0$  and the line $r_2$ obtained by the intersection of the planes \eqref{plane1} and \eqref{plane2}, whose projection on the plane $\{Z=0\}$ is
\begin{equation}\label{interm13}
Y=\frac{a}{c}X+\frac{d-b}{c}=\frac{3\sigma+7-m}{3(2\sigma+5-m)}X-\frac{m-1}{6(2\sigma+5-m)}=eX-f.
\end{equation}
The line $r_2$ intersects the plane $\{Y=0\}$ at the point $X=(m-1)/2(3\sigma+7-m)>X^*$, the latter inequality being equivalent to
$$
3\sigma^2+6(2-m)\sigma>2(m-1)(7-m)
$$
which is obviously true when $1<m<2$ and $\sigma>2$. Moreover, some easy calculations prove that fixing $X$ and $Z$ and letting $(X,Y_1,Z)$, respectively $(X,Y_2,Z)$ be the points on the line $r_1$, respectively $r_2$ for the $X$, $Z$ fixed, we have
$$
Y_2-Y_1=\left[\frac{3\sigma+7-m}{3(2\sigma+5-m)}-\frac{\sigma}{m-1}\right]X-\frac{m-1}{6(2\sigma+5-m)}+\frac{\sigma-2}{m-1}<0,
$$
provided $\sigma>2$ is sufficiently close to 2, giving that the line $r_2$ is more to the right than the line $r_1$. Moreover, the intersection between the plane \eqref{plane2} and the parabolic cylinder \eqref{cyl} is given by the curve inside \eqref{plane2} whose projection on the plane $\{Z=0\}$ writes
\begin{equation}\label{interm14}
X=\frac{1}{a}\left[Y^2+\frac{2(m-1)}{\sigma+2}Y+b\right]=:g(Y).
\end{equation}

\medskip

\noindent \textbf{Step 3.} Let us consider now the orbit going out of $P_2$. Consider first the region
$$
D_1=\left\{0\leq X\leq X^*, 0\leq Y\leq\frac{1}{2}, 0\leq Z\leq-cY+d \right\}
$$
with $c$, $d$ defined in \eqref{plane1}. The orbit going out of $P_2$ does this with
$$
X=X(P_2)=\frac{(m-1)^2(\sigma-2)}{2(m+1)(\sigma+2)}, \ \ Y(P_2)=\frac{(m-1)(\sigma-2)}{(m+1)(\sigma+2)}, \ \ Z=0.
$$
Thus, this orbit goes out into the region $D_1$ provided $\sigma$ is sufficiently close to 2 such that $X(P_2)<X^*$ and $Y(P_2)<1/2$. We infer from Lemma \ref{lem.monot} that the components $X$, $Y$ decrease along the orbit, thus these inequalities stay true all along the orbit. We also get from Step 1 above that, since at $P_2$ we have $Z=0<d-cY(P_2)=c(1/2-Y(P_2))$, we cannot cross the plane \eqref{plane1} since the direction of the flow on it is towards the interior of $D_1$. Thus the orbit going out of $P_2$ remains in the region $D_1$ until intersecting the plane $\{Y=0\}$. We now continue to "drive" the orbit going out of $P_2$ by considering the region
$$
D_2=\left\{0\leq X\leq X^*, eX-f\leq Y\leq0, -Y^2-\frac{2(m-1)}{\sigma+2}Y\leq Z\leq-cY+d \right\},
$$
where $e$, $f$ are defined in \eqref{interm13}. The orbit coming out of $P_2$ enters the region $D_2$ at $Y=0$. Indeed, we will always have $X\leq X^*$ by Lemma \ref{lem.monot} and it is obvious that $0<Z<d$ at $Y=0$ since we are just leaving the region $D_1$ introduced above. On the other hand Step 2 insures that, since $X<X^*$ and the intersection of the line $r_2$ introduced in \eqref{interm13} with the plane $\{Y=0\}$ occurs at a point $X>X^*$, we get $eX-f<0$ at $Y=0$. Once entered $D_2$, we show that our orbit stays in this region while $Y\geq eX-f$. We notice that $-f=Y^*$, thus while $Y\geq eX-f$ we are in particular in the region $\{Y>Y^*\}$ and the flow on the plane \eqref{plane1} continues to be negative, hence the inequality $Z\leq-cY+d$ is preserved along the trajectory. Finally, by Lemma \ref{lem.inner} and its proof it follows that the flow on the first part of the parabolic cylinder points from inside towards outside the cylinder, thus our orbit cannot cross it. This together with the monotonicity of $Z$ along the orbit preserve the inequality
$$
-Y^2-\frac{2(m-1)}{\sigma+2}Y\leq Z
$$
while we are in the region $D_2$. Unless entering one of the critical points $P_0^{\lambda}$ (which is our aim to prove), the orbit going out of $P_2$ can leave the region $D_2$ only when $Y=eX-f>-f=Y^*$. Let us finally consider the region
$$
D_3=\left\{0\leq X\leq X^*, g^{-1}(X)\leq Y\leq eX-f, -Y^2-\frac{2(m-1)}{\sigma+2}Y\leq Z\leq-aX+b\right\},
$$
where the function $g$ is defined in \eqref{interm14}. Since the line $Y=eX-f$ is the intersection between the planes \eqref{plane1} and \eqref{plane2}, if the orbit coming from $P_2$ leaves the region $D_2$ does so by entering the region $D_3$ which implies that $Z\leq-aX+b$. But we easily notice from Step 2 that for $\sigma$ sufficiently close to 2 the direction of the flow of the system on the plane $Z=-aX+b$ is pointing towards the interior of $D_3$ (as the line $r_2$ is more to the right in terms of $Y$ than the line $r_1$ where the direction of the flow changes on this plane) and thus the orbit cannot leave the region $D_3$ through its "wall" given by the plane \eqref{plane2}. Moreover, it cannot also leave the region $D_3$ by entering the interior of the parabolic cylinder \eqref{cyl} since the flow on it points towards outside as shown in the proof of Lemma \ref{lem.inner}. Finally, the orbit cannot leave the union $D_2\cup D_3$ through the planes $\{X=0\}$ or $\{Z=0\}$ as they are invariant for the system \eqref{PSsyst1}. We thus conclude that the orbit coming out of $P_2$, for $\sigma$ sufficiently close to 2, will remain forever in the set $D_2\cup D_3$. As we know that along this orbit component $Z$ is increasing and component $X$ is decreasing, it cannot form limit cycles and has to enter a critical point, thus it enters some of the points $P_0^{\lambda}$, as stated. Since the constructions above are highly geometric, we plot for the reader's convenience the regions $D_1$, $D_2$ and $D_3$ limiting the orbits going out of $P_2$ in Figure \ref{fig2}.

\begin{figure}[ht!]
  \begin{center}
  \includegraphics[width=11cm,height=8cm]{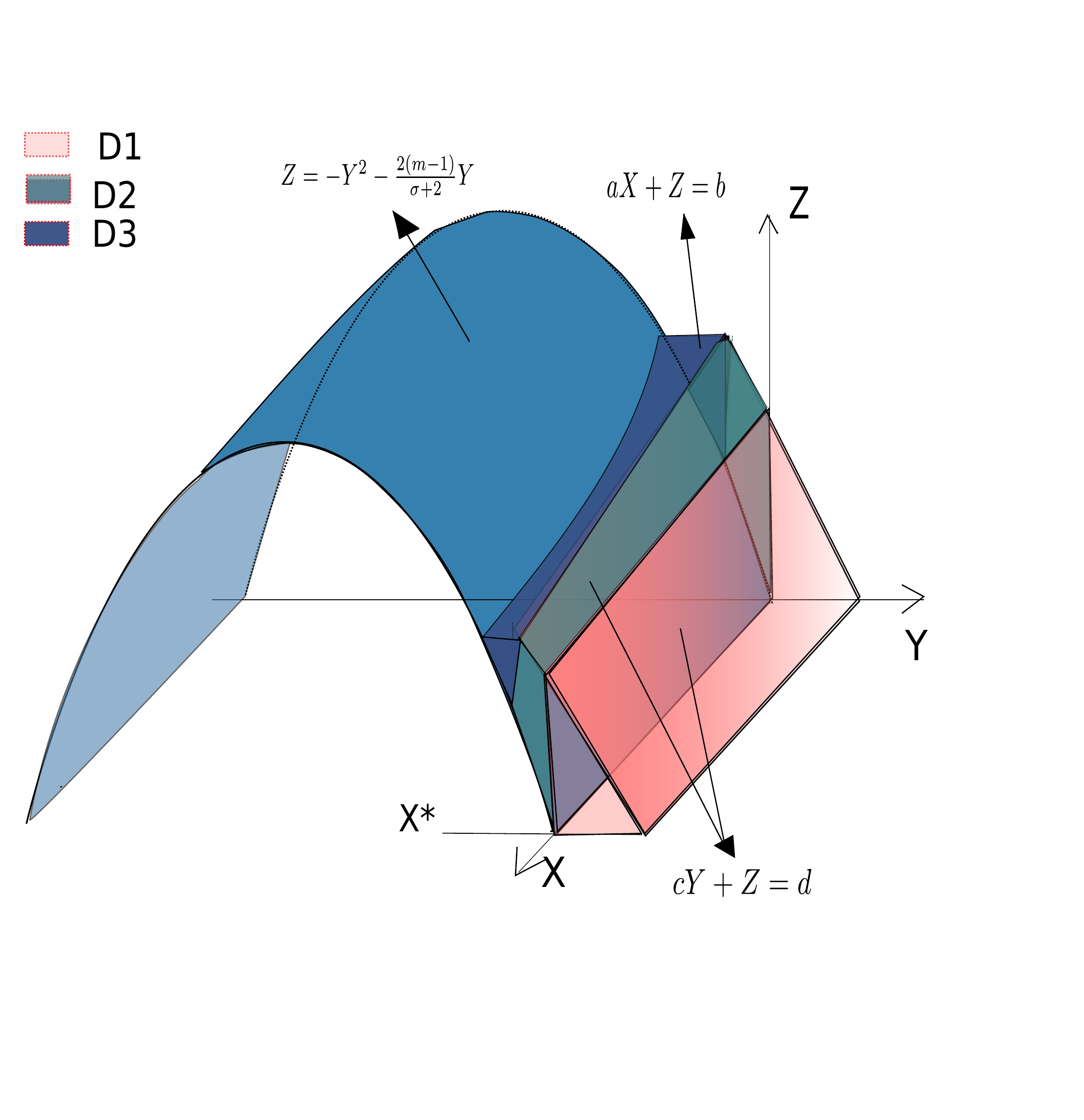}
  \end{center}
  \caption{A plot of the regions $D_1$, $D_2$ and $D_3$ in the phase space}\label{fig2}
\end{figure}

\medskip

\noindent \textbf{Step 4.} We show that for the same $\sigma\in(2,\sigma_0)$ for which all the above considerations hold true, all the trajectories from the critical point $P_0$ also enter some point $P_0^{\lambda}$. To this end, we repeat an argument already used for example in \cite[Proposition 4.1]{IS19a} and introduce the region
$$
D_4=\{0\leq X\leq X(P_2), 0\leq Y\leq Y(P_2)\},
$$
with $X(P_2)$, $Y(P_2)$ given in \eqref{not.P2}. The direction of the flow of the system \eqref{PSsyst1} on the plane $\{X=X(P_2)\}$ is given by the sign of the expression
$$
X(P_2)[(m-1)Y-2X(P_2)]\leq0, \qquad {\rm if} \ Y\leq Y(P_2),
$$
while the direction of the flow of the system \eqref{PSsyst1} on the plane $\{Y=Y(P_2)\}$ is given by the sign of the expression
\begin{equation*}
\begin{split}
E(X,Z)&=-Y(P_2)^2-\frac{2(m-1)}{\sigma+2}Y(P_2)+X(1-Y(P_2))-Z\\
&<X\left(1-\frac{1}{(m+1)\alpha}\right)-\frac{(m+1)\beta+1}{(m+1)^2\alpha^2}\\
&=\frac{1}{(m+1)^2\alpha^2}\left[\alpha(m+1)(\alpha(m+1)-1)X-\beta(m+1)-1\right],
\end{split}
\end{equation*}
which has negative sign, provided $X<(1+\beta(m+1))/[\alpha(m+1)(\alpha(m+1)-1)]=X(P_2)$. It thus follows that a connection going out in the region $D_4$ cannot leave the region $D_4$ unless by crossing the plane $\{Y=0\}$. Since all the orbits going out of $P_0=(0,0,0)$ do that in the region $D_4$, they will stay on $D_4$ and by taking $\sigma\in(2,\sigma_0)$ such that $X(P_2)<X^*$ and $Y(P_2)<1/2$, we conclude that the orbits going out of $P_0$ enter and remain in the region $D_1$ in Step 3 of the current proof until they cross the plane $\{Y=0\}$. Thus the above considerations apply to all these orbits too, proving that they have to enter one of the critical points $P_0^{\lambda}$.
\end{proof}
We plot in Figure \ref{fig3} the results of numerical experiments showing how the orbit going out of $P_2$ connects to one of the critical points $P_0^{\lambda}$ (including the vertex of the critical parabola) for values of $\sigma$ sufficiently close to 2, and also drawing what other orbits of the system in its neighborhood do. We can also see in the figure some connections coming from the points $P_0$ (appearing "below" the one from $P_2$) and $Q_1$ ("above" the one from $P_2$).

\begin{figure}[ht!]
  \begin{center}
  \subfigure[$\sigma$ close to 2]{\includegraphics[width=7.5cm,height=6cm]{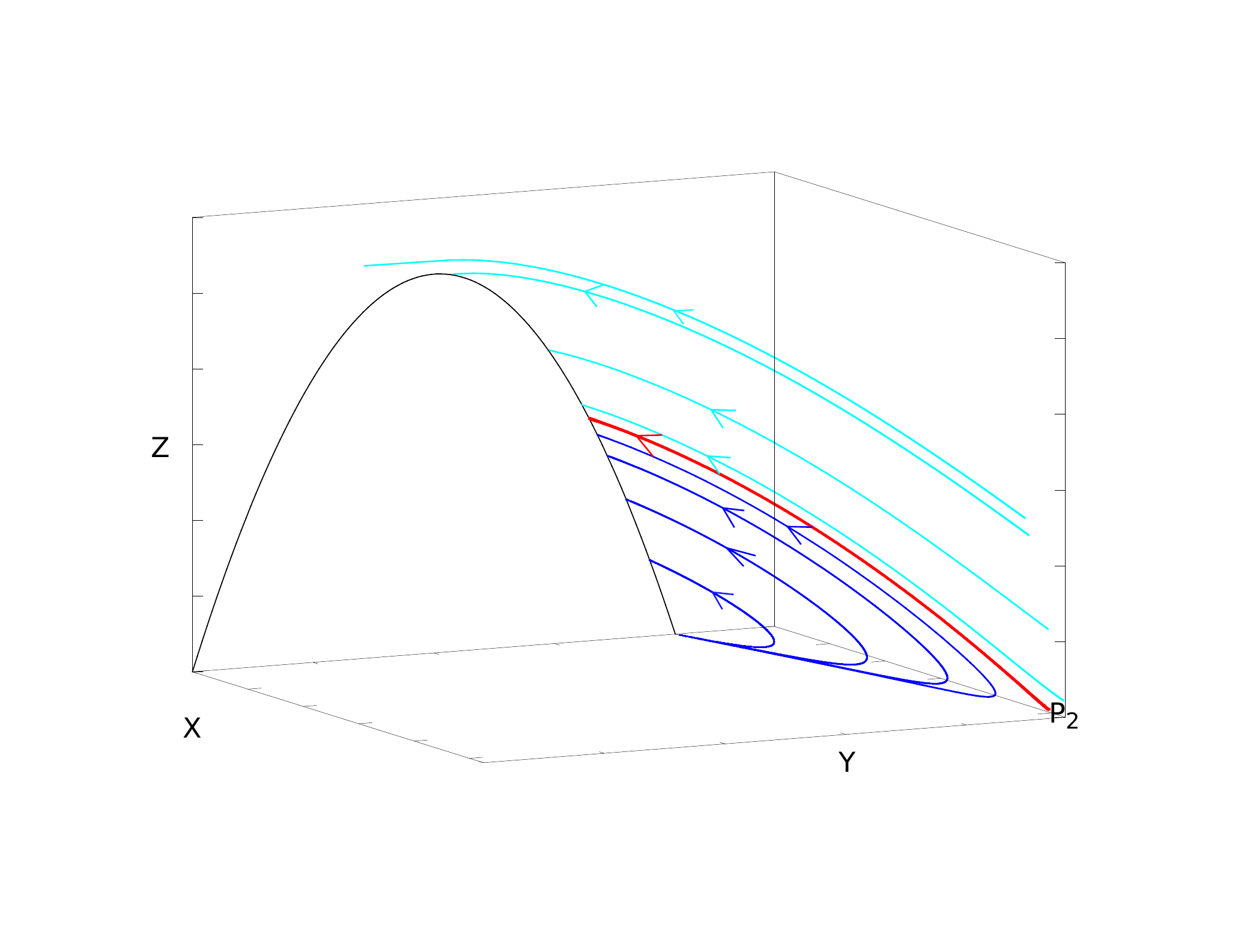}}
  \subfigure[Connection to the vertex]{\includegraphics[width=7.5cm,height=6cm]{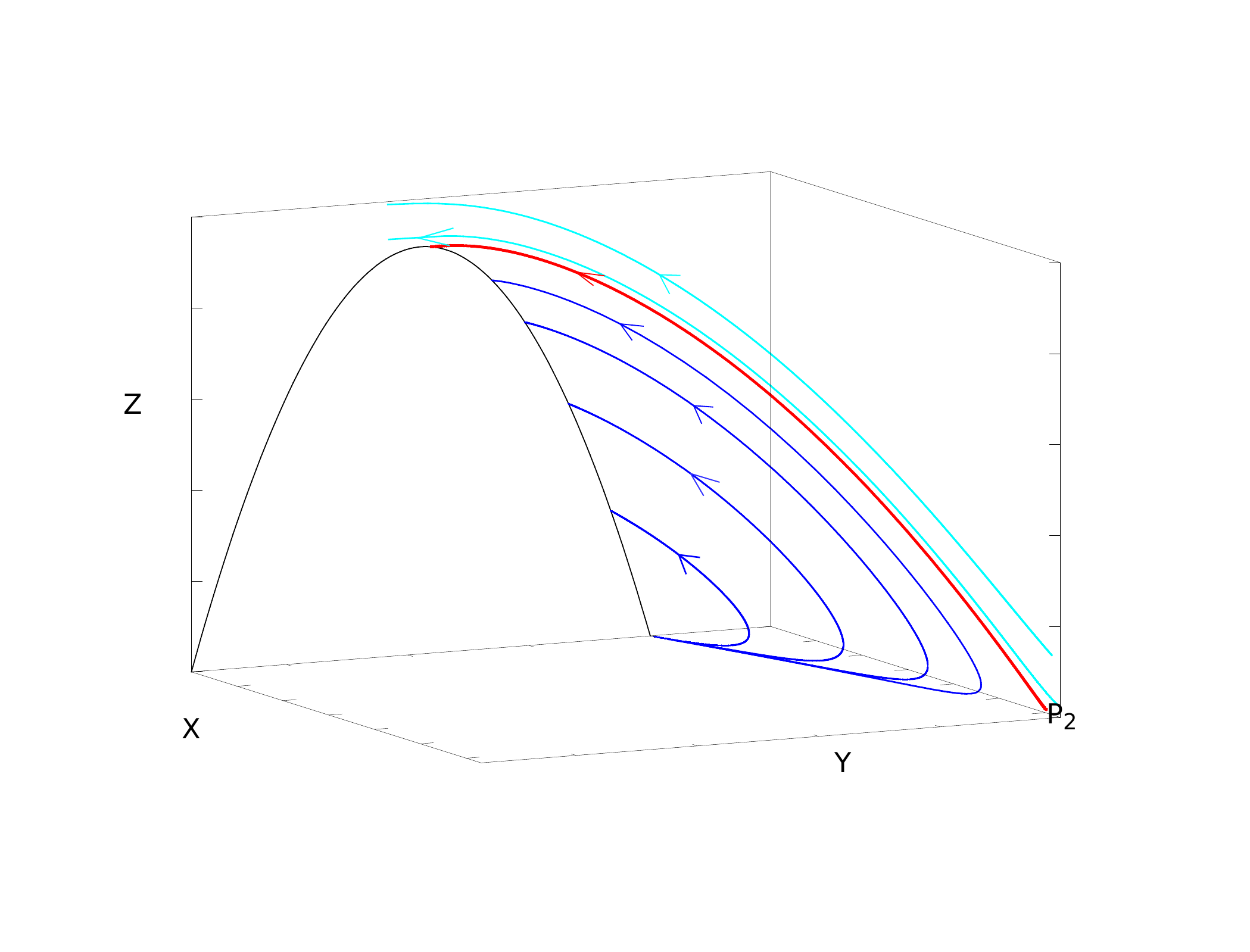}}
  \end{center}
  \caption{Orbits from $P_2$, $P_0$ and other points for different values of $\sigma$ close to 2. Numerical experiment for $m=1.5$, $p=0.5$ and $\sigma=3$, respectively $\sigma=3.285$}\label{fig3}
\end{figure}

On the contrary, when $\sigma>2$ is large, the orbit going out of $P_2$ does not enter any of the points $P_0^{\lambda}$.
\begin{proposition}\label{prop.large}
There exists $\sigma_1>2$ such that for any $\sigma\in(\sigma_1,\infty)$ the connection going out of $P_2$ enters the critical point $Q_3$ at infinity.
\end{proposition}
\begin{proof}
The idea of this proof is once more to drive the orbit from $P_2$ by a system of barriers but this time by constructing bounds from below for the component $Z$ in order to force it to increase more than the maximum value of the $Z$ coordinate on the critical parabola which is $Z=\beta^2/4\alpha^2$. We again divide the proof into several steps.

\medskip

\noindent \textbf{Step 1.} Let us consider the plane
\begin{equation}\label{plane3}
AX+BY+Z=C, \ A=\frac{(\sigma-1)(2m+\sigma)}{(\sigma+2)(m+1)}, \ B=\frac{(m-1)(2m+\sigma)}{(\sigma+2)(m+1)}, \ C=AX(P_2)+BY(P_2)
\end{equation}
where $X(P_2)$, $Y(P_2)$ are given in \eqref{not.P2}. Notice that $P_2$ belongs to this plane. The direction of the flow of the system \eqref{PSsyst1} on this plane is given by the sign of the scalar product of its normal direction $(A,B,1)$ with the vector field of the system, which after rather tedious calculations leads to
$$
G(X,Y)=B(Y(P_2)-Y)Y+A\sigma(X(P_2)-X)\left(X-X_{*}\right), \ \ X_{*}=\frac{(m-1)(\sigma+1)(2m+\sigma)}{\sigma(\sigma-1)(\sigma+2)(m+1)}
$$
which is positive provided that $Y<Y(P_2)$ and $X_*<X<X(P_2)$. We prove now that for $\sigma>2$ sufficiently large, the orbit coming out of $P_2$ goes above the plane \eqref{plane3}. To this end, we calculate the scalar product between the normal to the plane \eqref{plane3}, that is $\overline{n}=(A,B,1)$, with the eigenvector $e_3$ indicating the direction of the connection when going out of $P_2$. Recalling that $e_3$ is given in \eqref{interm.bis}, an easy calculation gives
$$
\overline{n}\cdot e_3=1-\frac{2(m+1)(m-1)(2m+\sigma)(\alpha\sigma-\alpha+\sigma)}{(\sigma+2)(m+1)[(m-1)\sigma^2+(5-m)\sigma+4m]}>0
$$
provided $\sigma$ is sufficiently large, since the limit of the above expression as $\sigma\to\infty$ is 1.

\medskip

\noindent \textbf{Step 2.} Let us now consider the region of the space
\begin{equation}\label{region}
R=\{X_{*}\leq X\leq X(P_2), 0\leq Y\leq Y(P_2), Z\geq C-AX-BY\}
\end{equation}
inside which the orbit from $P_2$ begins, according to Lemma \ref{lem.monot} and the last calculation in Step 1 for $\sigma$ sufficiently large. There is no critical point inside $R$ to which the connection might enter, neither in the plane nor at infinity. Since all the three components are monotonic along the orbit, it has to go out of the region $R$ and connect to a critical point. Since the flow on the plane \eqref{plane3} does not allow the orbit to cross it, the inequality $Z\geq C-AX-BY$ will hold true forever. We infer that the orbit, in order to quit the region $R$, has to do it either by crossing the plane $\{Y=0\}$ or by crossing the plane $\{X=X_{*}\}$. In the former, at the point where the orbit intersects the plane $\{Y=0\}$ we have
$$
Z\geq C-AX>C-AX(P_2)=BY(P_2)=\frac{(m-1)^2(2m+\sigma)(\sigma-2)}{(\sigma+2)^2(m+1)^2}>\frac{(m-1)^2}{(\sigma+2)^2}
$$
for $\sigma>2$ sufficiently large. Since the right hand side of the last inequality is the $Z$ component of the critical point $P_0^{\lambda}$ for $\lambda=-\beta/2\alpha$, which is the maximum value of $Z$ achieved by the points $P_0^{\lambda}$ and the $Z$ coordinate is strictly increasing along the orbits of the system \eqref{PSsyst1} we infer that the orbit going out of $P_2$ cannot enter any of the points $P_0^{\lambda}$. In the latter case, at the point where the orbit intersects the plane $\{X=X_{*}\}$ we have
$$
Z>C-AX_{*}-BY>C-AX_{*}-BY(P_2)=A(X(P_2)-X_*)\to\infty, \qquad {\rm as} \ \sigma\to\infty
$$
and a similar argument shows that the orbit cannot enter any of the points $P_0^{\lambda}$ provided $\sigma$ is large enough, ending the proof.
\end{proof}
We draw in Figure \ref{fig4} the outcome of Proposition \ref{prop.large} with the orbit going out of $P_2$ passing above the critical parabola for $\sigma$ sufficiently large. Some orbit coming from $Q_1$ and following the evolution of the orbit going out of $P_2$ are also shown in Figure \ref{fig4}.

\begin{figure}[ht!]
  \begin{center}
  \includegraphics[width=11cm,height=8cm]{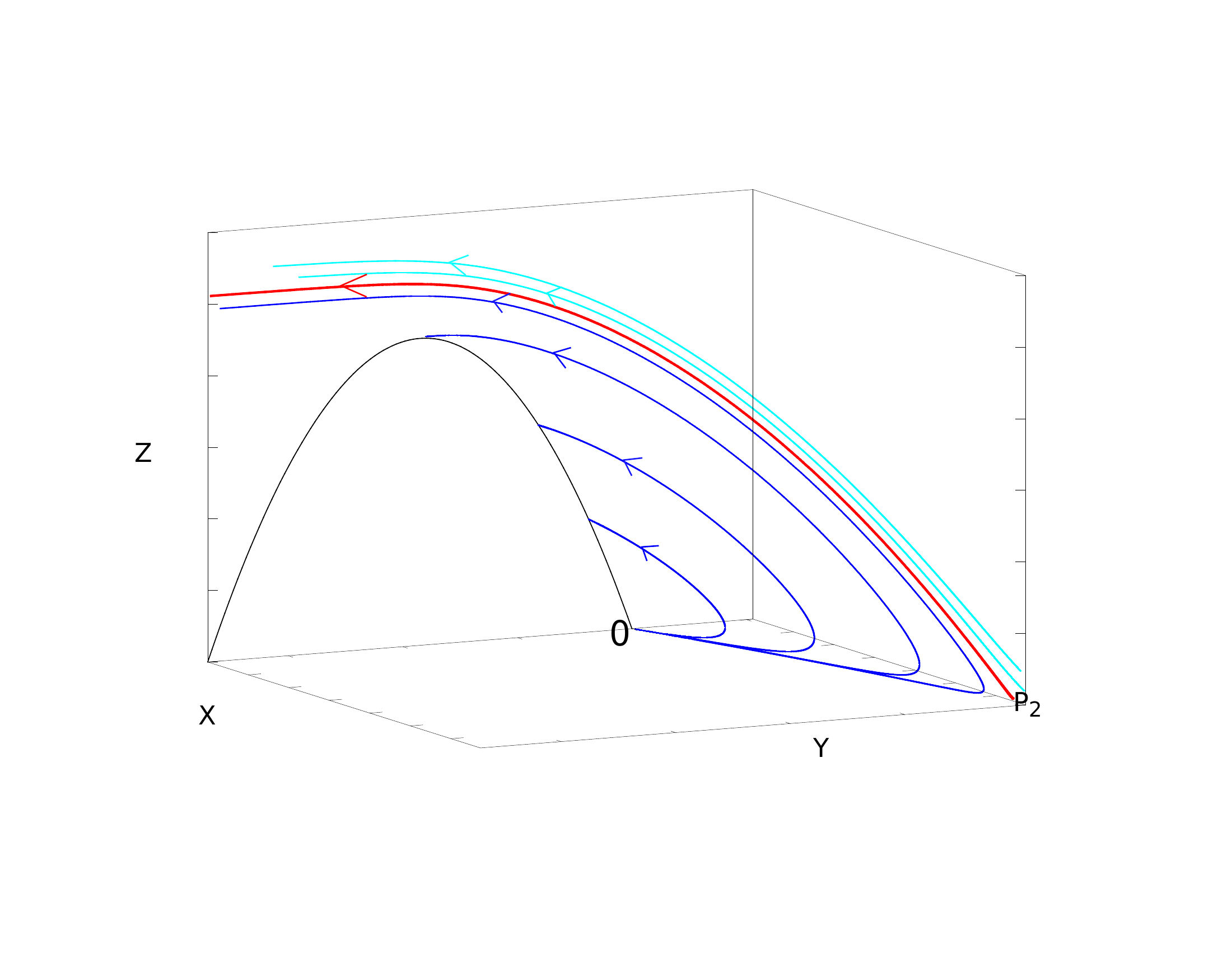}
  \end{center}
  \caption{The orbit from $P_2$ "avoiding" the critical parabola for $\sigma$ sufficiently large. Numerical experiment for $m=1.5$, $p=0.5$ and $\sigma=3.4$}\label{fig4}
\end{figure}

According to Propositions \ref{prop.small} and \ref{prop.large} and to the facts that the set $\mathcal{S}=\{P_0^{\lambda}: -\beta/2\alpha<\lambda<0\}$ is an asymptotically stable set by Proposition \ref{prop.att} and that $Q_3$ is an attractor, we infer that the sets of $\sigma\in(2,\infty)$ such that the orbit coming from $P_2$ enters $\mathcal{S}$, respectively $Q_3$ are both open and non-empty. Thus by the three-set argument we get that there exist (one or various) $\sigma\in(2,\infty)$ such that the unique orbit going out of $P_2$ enters the maximum point of the parabola $P_0^{-\beta/2\alpha}$ for such values of $\sigma$. Let $\sigma_*$ be the smallest of these values, thus the orbit going out of $P_2$ enters some point $P_0^{\lambda}$ in $\mathcal{S}$ for any $\sigma\in(2,\sigma_*)$. Let us denote by $\lambda(\sigma)$ the value of $\lambda$ for which the orbit coming from $P_2$ enters $P_0^{\lambda}$, for $\sigma\in(2,\sigma_*)$. We complete the fine global analysis of the orbits coming out of $P_2$ by the next result.
\begin{lemma}\label{lem.limit}
With the notation above we have
$$
\liminf\limits_{\sigma\to2}\lambda(\sigma)=0.
$$
\end{lemma}
\begin{proof}
Assume for contradiction that the conclusion is not true, then
$$
\liminf\limits_{\sigma\to2}\lambda(\sigma)=\lambda_0<0,
$$
which in particular means that the orbit coming out of $P_2$ does not connect to $P_0^{\lambda_0/2}$ for $\sigma$ at least in a right-neighborhood $\sigma\in(2,\sigma_0)$. The main idea of the proof is to show that on the one hand, the stable manifold of the point $P_0^{\lambda_0/2}$ is the "roof" to an invariant region of the phase space lying in the half-space $\{Y\leq0\}$, and on the other hand for $\sigma$ close to 2 the orbit going out of $P_2$ can cross the plane $\{Y=0\}$ as low (in terms of $Z$) as we wish, thus entering the invariant region and having to remain there. We divide the rest of the proof into several steps.

\medskip

\noindent \textbf{Step 1.} Consider the stable manifold of the point $P_0^{\lambda_0/2}$ and let $\gamma$ be the curve generated on the plane $\{Y=0\}$ by its intersection with this plane. On the one hand by Step 4 in the proof of Proposition \ref{prop.first} we infer that the curve $\gamma$ has an endpoint on the plane $\{X=0\}$, more precisely at the point $(0,0,Z)$ with $Z$ the same coordinate as for the point $P_0^{\lambda_0/2}$. On the other hand, since the plane $\{Y=0\}$ can be crossed by any point with $Z>X$, we easily infer that the other endpoint of the curve $\gamma$ lies on the line $\{X=Z\}$ inside the plane $\{Y=0\}$. We also notice that for $\sigma=2$ all the planes $Z={\rm constant}$ are invariant for the system \eqref{PSsyst1}, whence the curve $\gamma$ approaches the plane with $Z$ constant (equal to the $Z$ component of the point $P_0^{\lambda_0/2}$) as $\sigma\to2$. Let us now consider the region $D_0$ of the phase space to be the three-dimensional solid limited by the planes $\{Y=0\}$ and $\{X=0\}$, the surface $\mathcal{T}$ of equation $\{-Y^2-(\beta/\alpha)Y+X-XY-Z=0\}$ (that is, $\dot{Y}=0$) and the two-dimensional manifold of $P_0^{\lambda_0/2}$. We show that, once an orbit enters the region $D_0$, it cannot go out of it. To this end, we already know that an orbit cannot cross the stable manifold of $P_0^{\lambda_0/2}$ and the plane $\{X=0\}$ due to their invariance and also, once passed into the region $\{Y<0\}$ coming from the positive side, cannot go back to the region $\{Y>0\}$ due to the monotonicity of the components $X$ and $Z$ along the orbits in the region $\{Y\leq0\}$. It remains to study the flow of the system \eqref{PSsyst1} through the surface $\mathcal{T}$. The normal direction to this surface is given by the vector $(1-Y,-2Y-\beta/\alpha-X,-1)$, thus the direction of the flow is given by the sign of the scalar product between this vector and the vector field giving the system \eqref{PSsyst1}, namely
$$
(1-Y)X[(m-1)Y-2X]-(\sigma-2)XZ=X[(1-Y)(m-1)Y-2(1-Y)X-(\sigma-2)Z]<0,
$$
since all the terms in brackets are negative in the half-space $\{Y<0\}$. It follows that the flow on the surface $\mathcal{T}$ points towards the interior of the region $D_0$, showing that this region cannot be left by any of its "walls" by an orbit which entered it previously.

\medskip

\noindent \textbf{Step 2.} We prove that for $\sigma$ sufficiently close to 2, the orbit coming out of $P_2$ must enter the region $D_0$. To this end, let us consider the plane $\{Y+kZ=1\}$ for a $k>0$ to be determined later. It is obvious that the orbit going out of $P_2$ starts in the region $\{Y+kZ<1\}$ since $Y(P_2)<1$. The direction of the flow of the system \eqref{PSsyst1} on the plane $\{Y+kZ=1\}$ and in the region $\{X<X(P_2),Y>0\}$ is given by the sign of the expression
\begin{equation*}
\begin{split}
F(X,Y,Z)&=-Y^2-\frac{\beta}{\alpha}Y+X(1-Y)-Z+k(\sigma-2)Z\\
&=-Y^2-\frac{\beta}{\alpha}Y+k(\sigma-1)XZ-Z<-Y^2-\frac{\beta}{\alpha}Y+[k(\sigma-1)X(P_2)-1]Z,
\end{split}
\end{equation*}
which is negative provided, for example,
$$
k=\frac{2(m+1)\alpha}{(m-1)(\sigma-1)}.
$$
Since by Lemma \ref{lem.monot} the orbit going out of $P_2$ has $X<X(P_2)$ at any point on it, it follows that this orbit must remain in the region $\{Y+kZ<1\}$ at least until intersecting the plane $\{Y=0\}$. In particular, this orbit intersects the plane $\{Y=0\}$ at a point of coordinates
$$
X<Z<\frac{1}{k}=\frac{(m-1)(\sigma-1)}{2(m+1)\alpha}=\frac{(m-1)^2(\sigma-1)(\sigma-2)}{2(m+1)(\sigma+2)},
$$
which can be done as small as we want when $\sigma$ approaches 2. In particular, for $\sigma$ sufficiently small this orbit crosses the plane $\{Y=0\}$ below the curve $\gamma$ in Step 1, which approaches a constant positive value of $Z$, thus entering the region $D_0$ and remaining there according to Step 1. This is a contradiction to the fact that this orbit has to enter a critical point $P_0^{\lambda}$ with $\lambda\leq\lambda_0$, since all such points lie outside the region $D_0$, ending the proof.
\end{proof}
We next address the question of the orbits going out of $P_0$ on its center manifold as shown in Lemma \ref{lem.P1}. We already proved that for $\sigma$ sufficiently close to 2, all these orbits enter some of the critical points $P_0^{\lambda}$. But this happens for some of these orbits for any given $\sigma>2$.
\begin{proposition}\label{prop.P0}
For any $\sigma>2$, there exist orbits going out of $P_0$ and entering one of the critical points $P_0^{\lambda}$ in the phase space associated to the system \eqref{PSsyst1}.
\end{proposition}
\begin{proof}
By the study in Section \ref{sec.local} we deduce that all the orbits going out of $P_0$ on the center manifold \eqref{interm0} have to enter either one of the critical points $P_0^{\lambda}$ or the critical point $Q_3$. In order to control the orbits and show that some of them must choose one of the points $P_0^{\lambda}$, we consider the plane
\begin{equation}\label{plane4}
aX+Z=c, \qquad a=\frac{3}{(m-1)\alpha}, \ c>0.
\end{equation}
We work in the variables $(X,T,Z)$ characteristic for the center manifold as introduced in the proof of Lemma \ref{lem.P0}. The direction of the flow of the system \eqref{interm0} over this plane is given by the sign of the expression
\begin{equation*}
\begin{split}
H(X,T)&=\frac{a}{\beta}X[X+(m-1)\alpha T-(m-1)\alpha Z]+\frac{2}{\beta}XZ\\
&=\frac{1}{\beta}X[aX+(m-1)\alpha T-((m-1)\alpha a-2)(c-aX)]\\
&=\frac{1}{\beta}X[2aX-c+(m-1)\alpha T]
\end{split}
\end{equation*}
which is negative on the center manifold (where $T=T(X,Z)$ is a quadratic term) provided $X<X_0:=c/2a$. We next derive from the Local Center Manifold Theorem \cite[Theorem 1, Section 2.12]{Pe} that there exists a neighborhood $B(P_0,\delta)$ of the critical point $P_0$ where the orbits going out of $P_0$ on the center manifold behave as in the system \eqref{interm1}. Thus any orbit starting from $P_0$ on the center manifold will have an increasing component $X$ up to some maximum value, then the coordinate $X$ becomes decreasing while this orbit approaches the plane $\{Y=0\}$ , the maximum on $X$ being attained in the first order approximation on the plane $X=(m-1)\alpha Z$ after which the difference $X-(m-1)\alpha Z$ changes sign. Consider then the trace of the center manifold on the plane $X-(m-1)\alpha Z=0$. Let $c>0$ be sufficiently small such that $c/2a<\delta$. All the orbits on the center manifold intersecting the plane at points $(X,Z)$ with $0<X<X_0=c/2a$ (and consequently $Z=X/(m-1)\alpha<X_0/(m-1)\alpha$) satisfy
$$
aX+Z<aX_0+\frac{X_0}{(m-1)\alpha}=\frac{c}{2}+\frac{c}{2a(m-1)\alpha}=\frac{2c}{3}<c,
$$
whence these orbits go below the plane \eqref{plane4} and by the previous calculation, they cannot overpass it since $X$ will further decrease along the trajectory while $Z$ increases. These orbits will thus enter a critical point $P_0^{\lambda}$ with the component $Z>0$ very small (that is, $\lambda<0$ but sufficiently close to zero).
\end{proof}
We plot in Figure \ref{fig5} various orbits going out of $P_0$ on the center manifold \eqref{interm0} and entering critical points $P_0^{\lambda}$ as shown in the proof of Proposition \ref{prop.P0}.

\begin{figure}[ht!]
  \begin{center}
  \includegraphics[width=11cm,height=8cm]{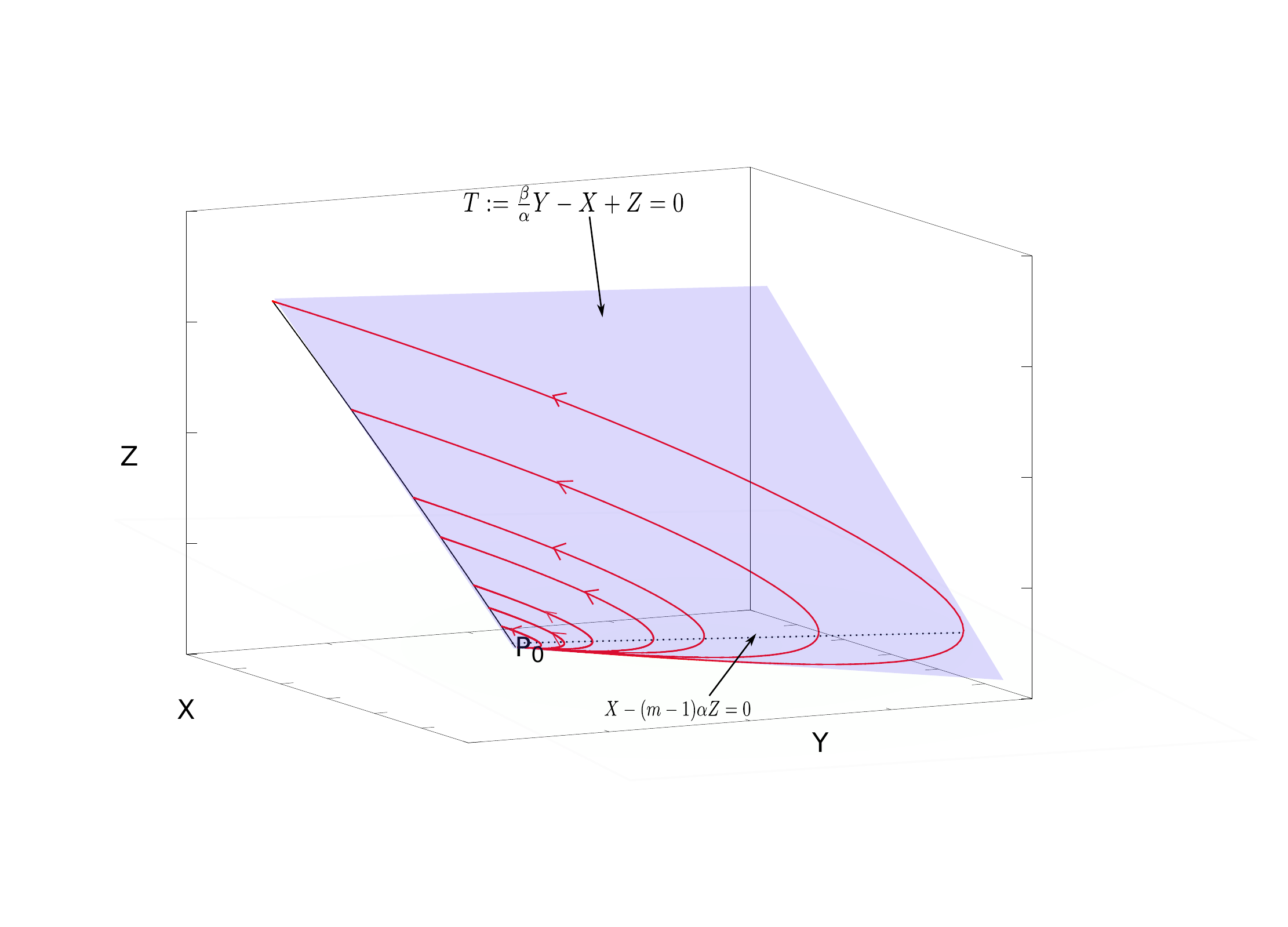}
  \end{center}
  \caption{Orbits going out of $P_0$ on the center manifold and entering critical points $P_0^{\lambda}$}\label{fig5}
\end{figure}

The following technical result taking place inside the invariant plane $\{Z=0\}$ is needed in order to prove afterwards the existence of good profiles with interface satisfying property (P1) in Definition \ref{def1}.
\begin{lemma}\label{lem.axis}
There exists a connection from $Q_1$ to $P_2$ inside the invariant plane $\{Z=0\}$ going out of $Q_1$ tangent to the line $\{Y=0\}$.
\end{lemma}
\begin{proof}
Let us recall first that, as shown in Lemma \ref{lem.P2}, the critical point $P_2$ is an attractor for the restriction of the system \eqref{PSsyst1} inside the invariant plane $\{Z=0\}$. The main difficulty is that the point $Q_1$ is a critical point at infinity, lying on the Poincar\'e hypersphere, making its analysis with barriers as in the previous Propositions quite difficult. To this end, recalling that $Q_1=(1,0,0,0)$ on the Poincar\'e hypersphere, we change the system using part (a) in \cite[Theorem 5, Section 3.10]{Pe} stating that the system \eqref{PSsyst1} in a neighborhood of $Q_1$ is topologically equivalent to the following system in new variables $(w,y,z)$ near the origin $(w,y,z)=(0,0,0)$ (that we will denote also by $Q_1$),
\begin{equation}\label{systinf1}
\left\{\begin{array}{ll}\dot{w}=w(2-(m-1)y),\\
\dot{y}=y+w-my^2-\frac{\beta}{\alpha}yw-zw,\\
\dot{z}=z(\sigma-(m-1)y),
\end{array}\right.
\end{equation}
where the new variables are given (with respect to our usual variables introduced in \eqref{PSvar1}) by
\begin{equation}\label{PSvar2}
w=\frac{1}{X}, \qquad y=\frac{Y}{X}, \qquad z=\frac{Z}{X}.
\end{equation}
We refer the reader to \cite[Lemma 3.3]{IS20b} for more details on the system \eqref{systinf1}. An easy calculation gives that the point $P_2$ becomes in these variables
$$
P_2=(w(P_2),y(P_2),0), \qquad w(P_2)=\frac{2(m+1)\alpha}{m-1}, \ y(P_2)=\frac{2}{m-1}.
$$
We will thus work with the system \eqref{systinf1} throughout this proof. The linearization of this system around the origin has eigenvalues and corresponding eigenvectors
$$
\lambda_1=2, \ \lambda_2=1, \ \lambda_3=\sigma, \qquad v_1=(1,1,0), \ v_2=(0,1,0), \ v_3=(0,0,1).
$$
Since $Q_1$ is a hyperbolic point (an unstable node) we infer by the Hartman Theorem \cite[Section 2.8, p.127]{Pe} and standard facts about linear systems that all the trajectories going out of $Q_1$ except for a two-dimensional sub-manifold go out tangent to the eigenvector corresponding to the smallest eigenvalue, that is $v_2$. It is easy to see that the profiles contained in these trajectories are the ones with $f(0)=A>0$ and $f'(0)\neq0$. There remains a two-dimensional unstable manifold tangent to the space spanned by the eigenvectors $v_1$ and $v_3$. Inside this manifold, all the trajectories but one go out tangent to the eigenvector $v_1$, since it corresponds to a smaller eigenvalue as always $2<\sigma$ in our study. These orbits going out tangent to the vector $v_1$ have a behavior given by $y=w$ in a neighborhood of the origin of the system \eqref{systinf1}, which in terms of profiles after undoing the changes of variables \eqref{PSvar2} and then \eqref{PSvar1} reads
$$
\frac{m}{\alpha}\xi^{-1}f^{m-2}(\xi)f'(\xi)\sim 1, \qquad {\rm as} \ \xi\to0,
$$
which after a direct integration gives
$$
f(\xi)\sim\left(A+\frac{\alpha(m-1)}{2m}\xi^{2}\right)^{1/(m-1)}, \qquad A>0 \ {\rm arbitrary},
$$
that is, good profiles according to Definition \ref{def1} with $f(0)=A^{1/(m-1)}>0$ and $f'(0)=0$. We also notice that the vector $v_1$ lies in the invariant plane $\{z=0\}$ of the system \eqref{systinf1}. We next restrict our study to this invariant plane, where the reduced system writes
\begin{equation}\label{systinf1.bis}
\left\{\begin{array}{ll}\dot{w}=w(2-(m-1)y),\\
\dot{y}=y+w-my^2-\frac{\beta}{\alpha}yw,
\end{array}\right.
\end{equation}
and prove that the orbit going out of $(0,0)$ tangent to the reduced vector $v_1=(1,1)$ connects to $P_2$. To this end, we consider the region $S$ in the phase plane associated to the system \eqref{systinf1.bis} limited by the $y$ axis and the following line and curve
\begin{equation}\label{interm15}
y=\frac{w}{\alpha(m+1)}, \ {\rm respectively} \ y+w-my^2-\frac{\beta}{\alpha}yw=0.
\end{equation}
Both the line and curve above connect the origin to the point $P_2$. We show that the region $S$ is invariant for the flow of the system \eqref{systinf1.bis}, that means, that any orbit entering the region $S$ cannot go out of it. This is done by inspecting the direction of the flow on the line and curve in \eqref{interm15}. On the one hand the normal direction to the line in \eqref{interm15} is given by the vector $(-1,\alpha(m+1))$ and the direction of the flow of the system \eqref{systinf1.bis} towards this line is given by the sign of the expression
\begin{equation*}
\begin{split}
F(w)&=-2w+\frac{m-1}{\alpha(m+1)}w^2-\frac{m}{\alpha(m+1)}w^2+w+\alpha(m+1)w-\frac{\beta}{\alpha}w^2\\
&=w\left[\alpha(m+1)-1-\frac{1+\beta(m+1)}{\alpha(m+1)}w\right]
\end{split}
\end{equation*}
which is positive for $w<2\alpha(m+1)/(m-1)=w(P_2)$. On the other hand the normal direction to the curve in \eqref{interm15} is given by the vector $(1-2(m-1)y/(\sigma+2),-2my+1-2(m-1)w/(\sigma+2))$ and the direction of the flow of the sytem \eqref{systinf1.bis} towards this curve is given by the sign of the expression
$$
G(w,y)=w(2-(m-1)y)\left(1-\frac{2(m-1)}{\sigma+2}y\right)
$$
which is positive for $y<2/(m-1)=y(P_2)$, since $\sigma+2>4$. Thus, the flow on the boundary of the region $S$ points towards the interior of the region, thus no orbit can go out of $S$. To ease the understanding of the technical details, we represent the region $S$ and the direction of the flow on and around its boundaries in Figure \ref{fig6}.

\begin{figure}[ht!]
  \begin{center}
  \includegraphics[width=11cm,height=8cm]{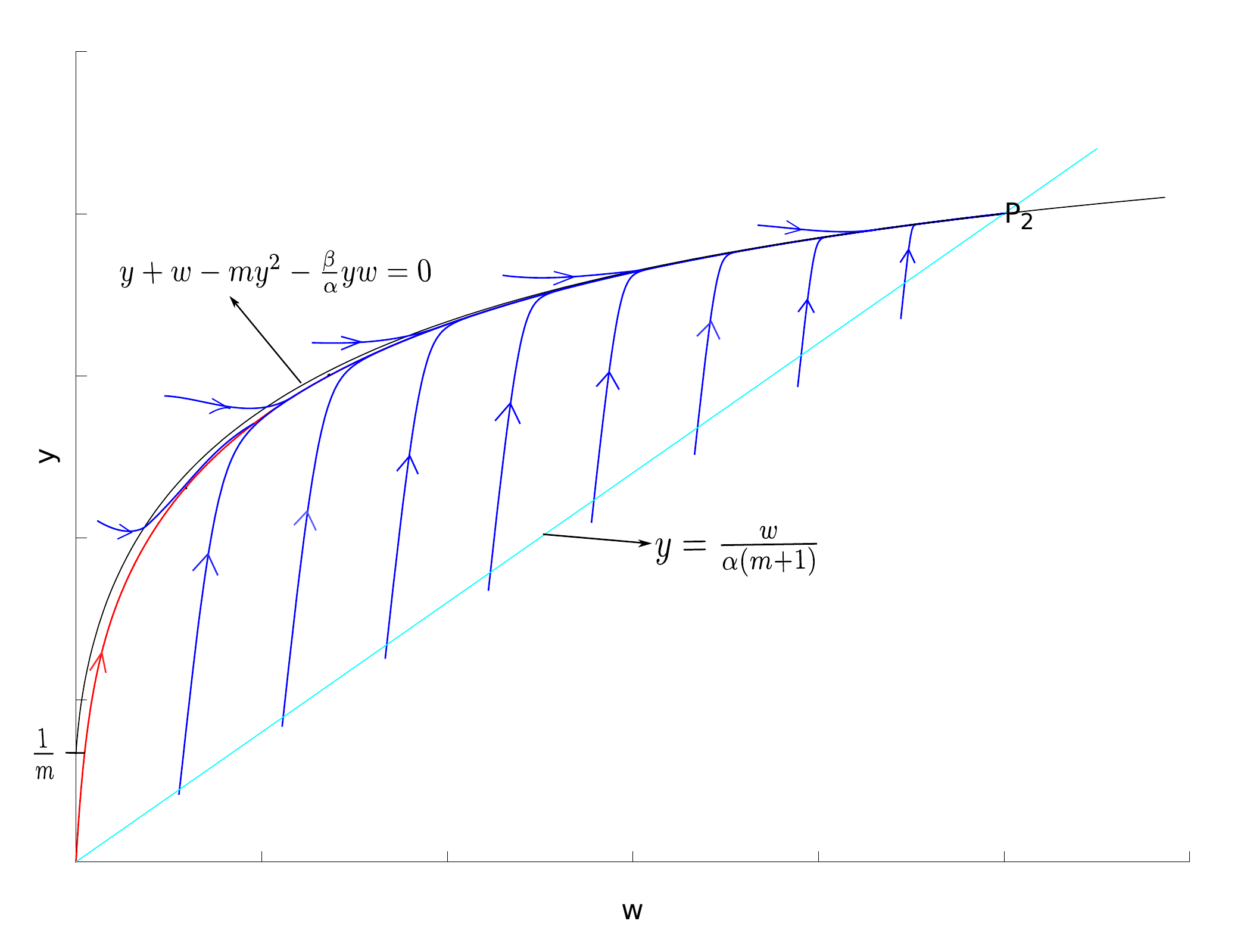}
  \end{center}
  \caption{The region $S$ and the orbit connecting $Q_1$ and $P_2$ in the plane $Z=0$}\label{fig6}
\end{figure}

It is not hard to check that the vector $v_1=(1,1)$ points towards the interior of $S$. Indeed, since $\alpha(m+1)>1$ for any $\sigma>2$, it follows that the slope of the line in \eqref{interm15} is strictly smaller than one, which is the slope of the line $y=w$. We thus infer that the orbit going out of $Q_1$ tangent to the vector $v_1=(1,1)$ enters $S$ and cannot go out of it later. Since its component $y$ is increasing, it ends by entering the critical point $P_2$. Undoing the change of variable \eqref{PSvar2} we conclude that the connection from $Q_1$ in the invariant plane $\{Z=0\}$ starting tangent to the line $\{Y=0\}$ enters $P_2$.
\end{proof}
\begin{proposition}\label{prop.Q1}
There exists $\sigma_0>2$ such that for any $\sigma\in(2,\sigma_0)$ there exists at least an orbit going out of $Q_1$, entering one of the critical points $P_0^{\lambda}$ and containing good profiles with interface.
\end{proposition}
\begin{proof}
Fix $\sigma>2$. We infer from Lemma \ref{lem.axis} and standard continuity arguments that there exists $\delta>0$ such that a connection going out of $Q_1$ tangent to the eigenvector $v_1$ intersect the half-ball $B(P_2,\delta)\cap\{Z>0\}$. But for $\delta>0$ sufficiently small, this half-ball is completely contained in the region
$$
D_1=\left\{0\leq X\leq X^*, 0\leq Y\leq\frac{1}{2}, 0\leq Z\leq-cY+d \right\}
$$
introduced in Step 3 of the proof of Proposition \ref{prop.small}. Taking $\sigma_0$ small exactly as in Proposition \ref{prop.small}, it follows as there that the orbit going out of $Q_1$ will enter one of the critical points $P_0^{\lambda}$.
\end{proof}
Let us notice here that some of the orbits coming out of the critical point $Q_1$ (the ones lying "above" the orbit coming from $P_2$) are represented in Figures \ref{fig3} and \ref{fig4}, with their expected behavior. We are now ready to obtain the classification theorem as an immediate consequence of the previous propositions.
\begin{proof}[Proof of Theorem \ref{th.class}]
\textbf{Part (a)} follows from Proposition \ref{prop.P0}, since any profile contained in the orbits going out of $P_0$ and entering some point $P_0^{\lambda}$ is a good profile with interface with local behavior as in \eqref{beh.P0} near the origin. \textbf{Part (b)} follows by joining the results of Propositions \ref{prop.small} (giving the good profiles with behavior as in \eqref{beh.P2} as $\xi\to0$), \ref{prop.P0} (giving the good profiles with behavior as in \eqref{beh.P0} as $\xi\to0$) and \ref{prop.Q1} (giving the good profiles with interface such that $f(0)=A>0$, $f'(0)=0$). All these three types of good profiles with interface exist at the same time for $\sigma\in(2,\sigma_0)$, where $\sigma_0>2$ is as in Proposition \ref{prop.small}.

In order to prove \textbf{part (c)}, we combine the outcome of Proposition \ref{prop.att} and Lemma \ref{lem.limit}. Indeed, let $\lambda_0\in(-\beta/2\alpha,0)$ be fixed, let $\sigma_*$ be the smallest $\sigma$ such that the orbit from the critical point $P_2$ enters the vertex of the critical parabola \eqref{parabola} and recall the notation $\lambda(\sigma)$ to be the $Y$ coordinate of the point on the critical parabola to which the orbit going out of the critical point $P_2$ enters. Define the sets
\begin{equation*}
\begin{split}
&A=\{\sigma\in(2,\sigma_*):\lambda(\sigma)\in(\lambda_0,0)\}, \ B=\{\sigma\in(2,\sigma_*):\lambda(\sigma)=\lambda_0\}, \\ &C=\{\sigma\in(2,\sigma_*):\lambda(\sigma)\in(-\frac{\beta}{2\alpha},\lambda_0)\}.
\end{split}
\end{equation*}
Proposition \ref{prop.att} gives that the sets $A$ and $C$ are both open, while Lemma \ref{lem.limit} insures that $A$ is a non-empty set. It is easy to see that the set $C$ must be also non-empty as a jump from $\lambda_0$ directly to $\lambda=-\beta/2\alpha$ is impossible as it can be easily seen by an argument completely similar to the one in Lemma \ref{lem.limit} that we leave to the reader.

Finally \textbf{part (d)} follows immediately from Proposition \ref{prop.large}.
\end{proof}

\bigskip

\noindent \textbf{Acknowledgements} A. S. is partially supported by the Spanish project MTM2017-87596-P.

\bibliographystyle{plain}

\end{document}